\documentclass[review,hidelinks,onefignum,onetabnum,dvipsnames]{siamart220329}

\usepackage{lipsum}
\usepackage{amsfonts}
\usepackage{graphicx}
\usepackage{epstopdf}
\usepackage{xcolor}
\usepackage{mathtools}
\ifpdf
  \DeclareGraphicsExtensions{.eps,.pdf,.png,.jpg}
\else
  \DeclareGraphicsExtensions{.eps}
\fi

\newsiamremark{remark}{Remark}
\newsiamremark{hypothesis}{Hypothesis}
\crefname{hypothesis}{Hypothesis}{Hypotheses}
\newsiamthm{claim}{Claim}
\newsiamthm{thm}{Theorem}
\newsiamthm{defi}{Definition}

\headers{Tree Tensor Network Nyström}{A. Bucci, G. Verzella}

\title{Randomized algorithms for streaming low-rank approximation in tree tensor network format}
\author{Alberto Bucci\thanks{Faculty of Mathematics and Physics, Charles University, Sokolovská 83, Prague, 186 75, CZ
  (\email{alberto.bucci@matfyz.cuni.cz}).}\and Gianfranco Verzella\thanks{Section of Mathematics, University of Geneva, CH-1205 Geneva, Switzerland} (\email{gianfranco.verzella@unige.ch})}

\usepackage{amsopn}

\ifpdf
\hypersetup{
  pdftitle={An Example Article},
  pdfauthor={D. Doe, P. T. Frank, and J. E. Smith}
}
\fi

\usepackage[shortlabels]{enumitem}
\usepackage{verbatim}
\usepackage{mathrsfs,mathtools,amscd,tikz-cd,dirtytalk,comment,mathabx}
\usepackage{bm}
\usepackage{xcolor}
\usepackage{algorithm}
\usepackage{pgfplots}
\usetikzlibrary{pgfplots.groupplots}
\usepackage{algpseudocode}
\usepackage{tkz-euclide}
\usepackage{tikz}
\usetikzlibrary{intersections}
\usepackage{siunitx}
\usepackage{subcaption}

\renewcommand{\hat}{\widehat}
\pgfplotsset{compat=1.18}

\begin{document}
\nolinenumbers
\maketitle
\begin{abstract}
In this work, we present the tree tensor network Nyström (TTNN), an algorithm that extends recent research on streamable tensor approximation, such as for Tucker and tensor-train formats, to the more general tree tensor network format, enabling a unified treatment of various existing methods. Our method retains the key features of the generalized Nyström approximation for matrices, that is randomized, single-pass, streamable, and cost-effective. Additionally, the structure of the sketching allows for parallel implementation. We provide a deterministic error bound for the algorithm and, in the specific case of Gaussian dimension reduction maps, also a probabilistic one. We also introduce a sequential variant of the algorithm, referred to as sequential tree tensor network Nyström (STTNN), which offers better performance for dense tensors. Furthermore, both algorithms are well-suited for the recompression or rounding of tensors in the tree tensor network format. Numerical experiments highlight the efficiency and effectiveness of the proposed methods.
\end{abstract}




 \section{Introduction}

Low-rank approximation has long been a cornerstone of numerical linear algebra. For instance, the singular value decomposition (SVD) enables optimal low-rank approximations of matrices by capturing their most significant singular components. 
Extending these ideas to multilinear arrays, or tensors, presents additional challenges due to the exponential growth of storage and computational costs with increasing dimensions. Various low-rank tensor decomposition methods, such as the Tucker format \cite{tensordecomp}, tensor-train (TT) format \cite{oseledets2011tensor}, and hierarchical Tucker (HT) format \cite{grasedyck2010hierarchical}, have been developed to mitigate this so-called curse of dimensionality. 

In this work, we focus on the more general tree tensor network (TTN) \cite{shi2006classical} format, which extends these classical tensor decompositions and has proven highly effective for various applications, including quantum chemical simulations \cite{murg2010simulating,murg2015tree}, dynamical low-rank approximation \cite{ceruti2021time, ceruti2023rank}, modeling quantum many-body systems with disorder \cite{lin2017griffiths}, and information science \cite{dumitrescu2017tree, cheng2019tree}.

This paper proposes and analyzes two randomized algorithms for the streaming low-rank approximation of tensors in the TTN format: the tree tensor network Nyström method (TTNN) and its sequential variant, the sequential tree tensor network Nyström (STTNN).

The foundation for developing randomized low-rank approximation algorithms is undoubtedly the randomized singular value decomposition, often referred to as the Halko-Martinsson-Tropp (HMT) method \cite{halko2011finding}.
Given a matrix $A\in \mathbb{R}^{m\times n}$, the algorithm proceeds by first drawing a random dimension reduction matrix (DRM) $X\in\mathbb{R}^{n\times r}$, with $r\ll n$, then by computing the product $AX$ and orthogonalizing its columns with the QR method which we denote by $Q = \mathrm{orth}(AX)$, and finally forming the rank $r$ approximation $Q(Q^TA)$. Many algorithms for tensor decomposition such as the higher-order SVD (HOSVD) \cite{hosvd}, the sequentially truncated higher-order SVD (STHOSVD)\cite{sthosvd}, and the tensor-train SVD (TT-SVD) \cite{oseledets2011tensor} have been greatly refined through the application of these randomized techniques \cite{Saibaba_Minster_Arvind, R-ST-HOSVD, rounding_TT, hashemi2023rtsms} and much of the current research is centered on the probabilistic analysis of these algorithms. In particular, the community has developed several tools and methodologies for conducting these analyses, providing rigorous guarantees of performance, accuracy, and stability \cite{userfriendly, Vershnynin_book,tropp2011improved,krahmer2011new,chen2024near, nakatsukasa2020fast}.

While these algorithms are highly effective, their reliance on the HMT framework necessitates at least two passes over the input data, making them unsuitable for streaming applications \cite{clarkson2009numerical}.

For matrices, the generalized Nyström method (GN) \cite{nakatsukasa2020fast, clarkson2009numerical,tropp2017practical} addresses this limitation. The algorithm draws two DRMs $X\in\mathbb{R}^{n\times r}$ and $Y\in\mathbb{R}^{m\times (r+ p)}$ and returns the rank $r$ approximant $AX(Y^TAX)^\dagger Y^TA$, where $\dagger$ denotes the Moore-Penrose pseudoinverse \footnote{In practice, the pseudoinverse is never computed, but an equivalent least-square problem is solved.}. The GN method achieves streamability because all the sketches depend linearly on $A$, which constitutes the primary computational expense. Furthermore, these sketches can be computed with just a single pass through the data. 
Building on the GN method, several approaches have been developed for streaming low-rank approximation tailored to specific matrix or tensor formats. These include
the one-pass sketch and low-rank recovery \cite{Another_Tucker}, the multilinear Nyström (MLN) \cite{bucci2024multilinear}, the sequential multilinear Nyström (SMLN) \cite{bucci2024sequential}, the TT-rounding two-sided-randomization \cite{rounding_TT}, and the streaming tensor-train approximation (STTA) \cite{kressner2023streaming}. 
To the best of our knowledge, no method with these properties has been proposed for the hierarchical format. This format is particularly advantageous due to its ability to achieve greater data compression than other formats, as it naturally adapts to the structure of the data. This adaptability results in approximations with lower ranks \cite{grasedyck2010hierarchical,ceruti2023rank}, enabling the efficient handling of high-dimensional tensors, which are commonly encountered in practical applications across physics, chemistry, biology, and mathematics.

 Additionally, our algorithms generalize and integrate all the aforementioned techniques into a unifying framework.

The structure of the paper is as follows. Section \ref{sec:randomized matrix} provides a recap of standard randomized techniques for the low-rank approximation of matrices, offering the foundational concepts necessary for the development and analysis of our methods. Section \ref{sec:TTN} introduces the TTN format, emphasizing its connections to other tensor formats such as Tucker and TT.
In section \ref{sec:TTNN}, we present the TTNN method, followed by section \ref{sec:STTNN}, where we describe its sequential variant. Section \ref{sec:error analysis} focuses on the error analysis of both methods. In particular, in section \ref{sec:deterministic}, a deterministic analysis of TTNN is provided, deriving general upper bounds for the accuracy that apply to any sketching. Specializing on Gaussian random DRMs, in section \ref{sec:probabilistic}, we derive an upper bound on the expected error of the TTNN method. Finally, in section \ref{sec:sequential deterministic}, a deterministic error bound for STTNN is furnished. 
In section \ref{sec:khatri-rao}, we explain how to use our algorithms to perform the rounding of a tensor in TTN format, by exploiting structured DRMs. 
In section \ref{sec:experiments}, we present numerical experiments, demonstrating the application of the algorithms for approximating tensors and their use for rounding. Finally, we conclude in section \ref{sec:conclusions} with a summary of our findings and a discussion of potential applications and future developments.

\section{Randomized matrix low-rank approximation}
\label{sec:randomized matrix}

The analysis of the tree tensor network Nyström method (TTNN) is based on results from the matrix case, which are briefly reviewed in this section. More specifically, we consider the approximant obtained by the HMT scheme from \cite{halko2011finding}
and the GN scheme from \cite{nakatsukasa2020fast}.

Given a matrix $A$ of size $m\times n$ the approximants obtained by HMT and GN methods are given respectively by
\[
    \hat{A}_{HMT} = Q (Q^T A)\quad \text{and} \quad \hat{A}_{GN} = AX (Y^T A X)^{\dagger} Y^T A,
\]
where $Q = \mathrm{orth}(AX)$, $X\in \mathbb{R}^{n\times r}$, $Y \in \mathbb{R}^{m\times (r+p)}$ are two DRM matrices, $r$ is the rank of the approximants and $p$ is an oversample parameter that improves accuracy and stability \cite{nakatsukasa2020fast}.

For any \(\hat{r} < r\), we denote by \(E_{SVD} := \|A - A_{\hat{r}}\|_F\) the Frobenius norm error of the best rank \(\hat{r}\) approximation \(A_{\hat{r}}\) of \(A\). According to the Eckart-Young theorem, this error satisfies  $E_{SVD} = \sqrt{\sum_{s > \hat{r}} \sigma_s^2(A)}$,
where \(\sigma_s(A)\) represents the \(s\)-th singular value of \(A\). Similarly, we denote the Frobenius norm error of the HMT approximant by \(E_{HMT}\) and that of the GN approximant by \(E_{GN}\). Then we have the following
upper bounds \cite{nakatsukasa2020fast}
\begin{align} \label{eq: HMT_error_bound}
E_{HMT} &\leq E_{SVD} \|(\hat{V}_{\perp}^{T}X) (\hat{V}^T X)^\dagger\|_2,\\ \label{eq: GN_error_bound}
\quad E_{GN} &\leq E_{HMT}\|(Q_{\perp}^T Y)(Q^T Y)^\dagger\|_2,
\end{align}
where $\hat{V}$ is the orthogonal matrix with the first $\hat{r}$ right singular vectors of $A$, ${M}_\perp$ denotes an orthogonal complement of an orthogonal matrix ${M}$, and $\|\cdot\|_2$ denotes the 2-norm. 
Since $E_{SVD}$ is the optimal error that would be obtained by a truncated SVD, it is clear that it is important to choose the DRM in a way that makes the other terms as small as possible (with high probability). At the same time, we wish to maintain the cost of taking the matrix-vector products small, so it makes sense to use DRMs drawn from a set of structured matrices that have fast matrix-vector product routines available. A choice of random sampling that enables fast multiplication is that of sparse DRMs such as CountSketch matrices \cite{clarkson2017low}. This approach reduces the cost of forming $AX$ to $\mathcal{O}(\text{nnz}(A))$, where
$\text{nnz}(A)$ is the number of non-zero entries of $A$. Other options arise from subsampling trigonometric transforms. Examples include the
Subsampled Randomized Hadamard Transform (denoted by SRHT) \cite{tropp2011improved,boutsidis2013improved}
and the Subsampled Randomized Fourier Transform (SRFT) \cite{rokhlin2008fast}. These approaches reduce the 
cost of forming $AX$ to $\mathcal{O}(mn \log n)$, where $m$ is the number of rows of $A$, and $n$ is the number of 
rows of $X$. The theory for these transforms can be more complex than the one for more ``classical'' choices, such as Gaussian
matrices; the latter are deeply understood and have sharp error bounds available (see \cite{randomization_advantajes} and the references therein).
In particular, when the DRMs in the HMT and GN are Gaussian matrices we have the following bounds in expectation
\begin{align} \label{eq: HMT_error_bound_expecation}
\mathbb{E}[E_{HMT}] &\leq E_{SVD}\cdot\sqrt{1+\frac{\hat{r}}{r-\hat{r}-1}},\\ \label{eq: GN_error_bound_expectation}
\quad \mathbb{E}[E_{GN}] &\leq E_{SVD}\cdot\sqrt{1+\frac{\hat{r}}{r-\hat{r}-1}}\sqrt{1 + \frac{r}{p - 1}}. 
\end{align}
The use of GN has a few advantages with respect to the HMT scheme: 
it avoids costly orthogonalizations and can be used as a single-pass approximation
method. However, without proper implementation, the stability of GN can be cause for concern. The pseudocode in Algorithm~\ref{alg:GN}
reports the implementation suggested in \cite{nakatsukasa2020fast}.
\begin{algorithm}
\caption{Generalized Nystr\"om (GN)}\label{alg:GN}
\begin{algorithmic}[1]
    \Require $A\in \mathbb{R}^{m\times n}$, $r$ rank of the desired approximation, $p$ oversampling parameter,
    \Ensure Rank $r$ approximation of $A$.
    \State Draw two random sketchings $X\in \mathbb{R}^{n\times r}$ and $Y\in\mathbb{R}^{m\times (r+p)}$.
    \State Compute $AX$, $Y^TAX$, and $Y^TA$.
    \State Compute the economy-size QR decomposition $QR = Y^TAX$. 
    \State Return $(AXR^{-1})(Q^TY^TA)$.
\end{algorithmic}
\end{algorithm}
In practice, a slight oversample parameter $p$ makes this implementation stable, but no theoretical assessments have been conducted. While stability cannot be established
for GN as is, there is an inexpensive modification that guarantees stability
\begin{equation} \label{stabilized_nystrom}
\hat{A} = AX (Y^TAX)^{\dagger}_\epsilon Y^TA,
\end{equation}
which is the stabilized generalized Nyström (StabGN) method. Here $(Y^TAX)_\epsilon^\dagger$ denotes the $\epsilon$-pseudoinverse, that is if 
\[Y^TA X =\begin{bmatrix}
    U_1 & U_2
\end{bmatrix}\begin{bmatrix}
    \Sigma_1 & \\
    & \Sigma_2
\end{bmatrix}\begin{bmatrix}
    V_1 & V_2
\end{bmatrix}^T\]
is the SVD, where $\Sigma_1$ contains singular values larger than $\epsilon$, then $(Y^TAX)^\dagger_\epsilon = V_1 \Sigma_1^{-1} U_1^T$. Different strategies to implement StabGN in a numerically stable manner can be found in \cite{nakatsukasa2020fast}.

 \section{Tree tensor network}\label{sec:TTN}
A \textit{tensor} 
$\mathcal{T}\in\mathbb{R}^{n_1\times \dots \times n_d}$ is a $d$ dimensional array with entries
\[
t_{i_1i_2\dots i_d}, \quad 1\leq i_k \leq n_k, \quad  k \in D := \{1, \dots, d\}.
\]

\begin{defi}[mode-$k$ product]
The $k$-mode product of tensor $\mathcal{T}\in \mathbb{R}^{n_1\times \dots \times n_d}$ with tensor $\mathcal{U}\in \mathbb{R}^{n_k \times m_1\times \dots \times m_s}$, denoted $\mathcal{T}\times_k \,\mathcal{U}\in \mathbb{R}^{n_1\times \dots \times n_{k-1}\times m_1 \times \dots \times m_s\times n_{k+1}\times \dots \times n_{d}}$, is the contraction of the $k$th index of $\mathcal{T}$ and the first index of $\mathcal{U}$. Elementwise, we have
\[
(\mathcal{T}\times_{k}\mathcal{U})_{i_1,\dots, i_{k-1}, j_1,\dots, j_s, i_{k+1}, \dots, i_d} = \sum_{i_k=1}^{n_k} \mathcal{T}_{i_1,\dots ,i_d}\mathcal{U}_{i_k,j_1,\dots, j_s}.
\]
\end{defi}
Typically, the mode-$k$ product involves multiplying a tensor by a matrix along the $k$th mode. However, for our purposes, it is preferable to adopt this broader definition, that is the classical mode-$k$ product with a proper matricization of $\mathcal{U}$. We will also denote the contraction over the set of indices \( I=\{i_1,\dots, i_s\} \) of \( \mathcal{T} \) and the first $s$ indices of \( \mathcal{U} \) as \( \mathcal{T} \times_I \mathcal{U} \). It is important to note that when \( I = \{1, \dots, s\}\), this notation differs from \( \mathcal{T} \times_{i=1}^s \mathcal{U}_i \), where each mode of \( \mathcal{T} \) is contracted with the first index of \( \mathcal{U}_i \). Given a subset of indices $I\subset D$ let $n_I := \prod_{i\in I} n_i$ and $n_{D\backslash I} := \prod_{i\in D\backslash I} n_i$.

We will repeatedly use the unfolding operation, or matricization, which reshapes tensors into matrices. 
\begin{defi}[Matricization]    Consider a tensor $\mathcal{T}\in \mathbb{R}^{n_1\times \dots \times n_d}$. Let $I = \{\alpha_1, \dots, \alpha_k\}$ be a subset of its indices and let $D\backslash I = \{\beta_1, \dots, \beta_{d-k}\}$ be its complementary, both ordered in increasing order. The \textit{mode}-$I$ \textit{matricization} of $\mathcal{T}$, denoted by $\mathcal{T}^I\in \mathbb{R}^{n_I\times n_{D\backslash I}}$, satisfies

    \[
    {(\mathcal{T}^{I})}_{(i, j)} = t_{i_1,\dots, i_d},
    \]

    where
    \[
    i =  1 + \sum_{t=1}^k (i_{\alpha_t}-1)N_{\alpha_t}, \quad N_{\alpha_t} = \prod_{s=1}^{t-1} n_{\alpha_s}
    \]

    and
    \[
    j =  1 + \sum_{t=1}^{d-k} (i_{\beta_{t}}-1)M_{\beta_{t}}, \quad M_{\beta_{t}} = \prod_{s=1}^{t-1} n_{\beta_{s}}.
    \]
\end{defi}
Next definition is similar to the one of dimension tree given in \cite{grasedyck2010hierarchical}. 
\begin{defi}[index tree] \label{def:index_tree}
    Given a set of indices $D=\{1,\dots,d\}$, a family on nodes $\mathcal{I} = \{I_{\ell, k}\}$, for $\ell\in \{0,\dots, L\}$ and $k\in \{1, \dots, K_{\ell}\}$, where each $I_{k,\ell}$ is a subset of indices of the tensor, is said to be an index tree with root $I_{0,1}=D$ if each node $I_{\ell, k}$ satisfies one of the following properties:
    \begin{itemize}
    \item is a leaf (it has no successors),  
    \item it contains the union of $m_{\ell,k}$ disjoint successors (also referred to as children)
    \begin{equation*}
        I_{\ell,k}\supseteq  \dot{\bigcup}_{t\in c(I_{\ell,k})} t,
    \end{equation*}
    where $c(I_{\ell,k})$ is the set of the children of $I_{\ell,k}$.
    \end{itemize}
\end{defi}

The primary advantage of this notation is that each node of the tree can be accessed using only two indices. To simplify further, in expressions involving \( I_{\ell,k} \), we will use only \( \ell \) and \( k \). For example, we denote \( \mathcal{T}^{I_{\ell,k}} \) simply as \( \mathcal{T}^{\ell,k} \). 
Note that Definition \ref{def:index_tree} differs from the definition of dimension tree for hierarchical Tucker \cite[Definition~3.1]{grasedyck2010hierarchical} in the number of children per node and the depth of the leaves. Our definition reduces to  \cite[Definition~3.1]{grasedyck2010hierarchical} if the index tree is almost a complete binary tree, except that on the last but one
level there may appear leaves. An example of index tree is shown in Figure \ref{fig: example of index tree}. This tree serves as a toy example, which we will use frequently throughout our discussion to elucidate various aspects of the structure and relationships involved.

\begin{figure}
    \centering
\begin{tikzpicture}

     \node (D1) at (6.5,2) {$I_{0,1}$};
    \node (E1) at (5.3,1) {$I_{1,1}$};
    \node (E2) at (6.5,1) {$I_{1,2}$};
    \node (E3) at (7.7,1) {$I_{1,3}$};
    \node (F1) at (4.6,0) {$I_{2,1}$};
    \node (F2) at (6.0,0) {$I_{2,2}$};
    \node (F3) at (7,0) {$I_{2,3}$};
    \node (F4) at (8.4,0) {$I_{2,4}$};
    \node (G1) at (3.9,-1) {$I_{3,1}$};
    \node (G2) at (5.3,-1) {$I_{3,2}$};

    \draw (D1) -- (E1);
    \draw (D1) -- (E2);
    \draw (D1) -- (E3);
    
    \draw (E1) -- (F1);
    \draw (E1) -- (F2);
    \draw (E3) -- (F3);
    \draw (E3) -- (F4);

    \draw (F1) -- (G1);
    \draw (F1) -- (G2);

\node (M) at (9.4,0.7) {$=$};

\node (AD1) at (12.5,2) {$\{1,2,3,4,5,6\}$};
\node (AE1) at (11.3,1) {$\{1,2,3\}$};
\node (AE2) at (12.5,1) {$\{4\}$};
\node (AE3) at (13.7,1) {$\{5, 6\}$};
\node (AF1) at (10.6,0) {$\{1,2\}$};
\node (AF2) at (12,0) {$\{3\}$};
\node (AF3) at (13,0) {$\{5\}$};
\node (AF4) at (14.4,0) {$\{6\}$};
\node (AG1) at (9.9,-1) {$\{1\}$};
\node (AG2) at (11.3,-1) {$\{2\}$};

\draw (AD1) -- (AE1);
\draw (AD1) -- (AE2);
\draw (AD1) -- (AE3);

\draw (AE1) -- (AF1);
\draw (AE1) -- (AF2);
\draw (AE3) -- (AF3);
\draw (AE3) -- (AF4);

\draw (AF1) -- (AG1);
\draw (AF1) -- (AG2);

\end{tikzpicture}
    \caption{An example of an index tree.}
    \label{fig: example of index tree}
    \end{figure}
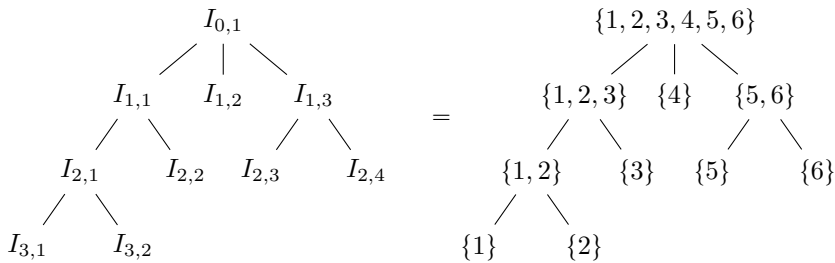

Since defining the TTN format starting from the definition of an index tree is rather complicated, we prefer to first provide the definition of an extended index tree. As we will see, it will be easier to define the format once this definition is given.
\begin{defi}{(extended index tree).}
    Given an index tree $\mathcal{I}$, the \emph{extended index tree} $\overline{\mathcal{I}}$ of $\mathcal{I}$ is the tree obtained from $\mathcal{I}$ by adding an empty node, called \textit{dummy node}, to each leaf that is at level $1\leq\ell<L$ until all leaves are at level $L$. Some quantities, such as the number of nodes and children, may change when using the extended index tree. 
\end{defi}
Figure \ref{fig: example of extended index tree} shows the extended index tree of the index tree in Figure \ref{fig: example of index tree}. 

\begin{figure}
    \centering
\begin{tikzpicture}

     \node (D1) at (6.5,2) {$\{1,2,3,4,5,6\}$};
    \node (E1) at (5.0,1) {$\{1,2,3\}$};
    \node (E2) at (6.5,1) {$\{4\}$};
    \node (E3) at (8,1) {$\{5, 6\}$};
    \node (F1) at (4.3,0) {$\{1,2\}$};
    \node (F2) at (5.6,0) {$\{3\}$};
    \node (F3) at (7.4,0) {$\{5\}$};
    \node (F4) at (8.4,0) {$\{6\}$};
    \node (F5) at (6.5,0) {$\{\phantom{4}\}$};
    \node (G1) at (3.8,-1) {$\{1\}$};
    \node (G2) at (4.8,-1) {$\{2\}$};
    \node (G3) at (5.6,-1) {$\{\phantom{3}\}$};
    \node (G4) at (6.5,-1) {$\{\phantom{4}\}$};
    \node (G5) at (7.4,-1) {$\{\phantom{4}\}$};
    \node (G6) at (8.4,-1) {$\{\phantom{4}\}$};

    \draw (D1) -- (E1);
    \draw (D1) -- (E2);
    \draw (D1) -- (E3);
    
    \draw (E1) -- (F1);
    \draw (E1) -- (F2);
    \draw (E3) -- (F3);
    \draw (E3) -- (F4);

    \draw (F1) -- (G1);
    \draw (F1) -- (G2);

    \draw (E2)--(F5)--(G4);
    \draw (F2)--(G3);
    \draw (F3)--(G5);
    \draw (F4)--(G6);
     \node (AD1) at (0.5,2) {$\overline{I_{0,1}}$};
    \node (AE1) at (-1.0,1) {$\overline{I_{1,1}}$};
    \node (AE2) at (0.5,1) {$\overline{I_{1,2}}$};
    \node (AE3) at (2.0,1) {$\overline{I_{1,3}}$};
    \node (AF1) at (-1.7,0) {$\overline{I_{2,1}}$};
    \node (AF2) at (-0.4,0) {$\overline{I_{2,2}}$};
    \node (AF3) at (1.4,0) {$\overline{I_{2,4}}$};
    \node (AF4) at (2.4,0) {$\overline{I_{2,5}}$};
    \node (AF5) at (0.5,0) {$\overline{I_{2,3}}$};
    \node (AG1) at (-2.2,-1) {$\overline{I_{3,1}}$};
    \node (AG2) at (-1.2,-1) {$\overline{I_{3,2}}$};
    \node (AG3) at (-0.4,-1) {$\overline{I_{3,3}}$};
    \node (AG4) at (0.5,-1) {$\overline{I_{3,4}}$};
    \node (AG5) at (1.4,-1) {$\overline{I_{3,5}}$};
    \node (AG6) at (2.4,-1) {$\overline{I_{3,6}}$};
     \node (EQ) at (3.4,0.7) {$=$};
    
    \draw (AD1) -- (AE1);
    \draw (AD1) -- (AE2);
    \draw (AD1) -- (AE3);
    
    \draw (AE1) -- (AF1);
    \draw (AE1) -- (AF2);
    \draw (AE3) -- (AF3);
    \draw (AE3) -- (AF4);

    \draw (AF1) -- (AG1);
    \draw (AF1) -- (AG2);

    \draw (AE2)--(AF5)--(AG4);
    \draw (AF2)--(AG3);
    \draw (AF3)--(AG5);
    \draw (AF4)--(AG6);

\end{tikzpicture}
    \caption{Extended index tree of the index tree in Figure \ref{fig: example of index tree}.}
    \label{fig: example of extended index tree}
    \end{figure}
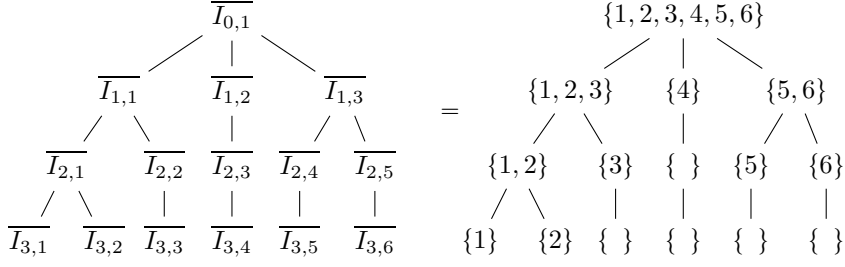

\begin{defi}[Tree tensor network format]\label{def: tree tensor network format}
    A tree tensor network representation of  $\mathcal{T}\in\mathbb{R}^{n_1\times \dots \times n_d}$ of index tree $\mathcal{I}$ consists of a family of tensors $\{\mathcal{B}^{(\ell, k)}\}_{\overline{I_{\ell,k}}\in\Bar{\mathcal{I}}}$ assembled in the following way

    \begin{equation}\label{eq: tree_tensor_def}
         \mathcal{T} = \mathcal{B}^{(0,1)}\times_{k_0=1}^{m_{0,1}}\left(\mathcal{B}^{(1,k_0)}\times_{k_1=1}^{m_{1,k_0}}\left(\mathcal{B}^{(2,k_1)}\dots \left( \mathcal{B}^{(L-1,k_{L-1})}\times_{k_{L-1}=1}^{m_{L-1,k_{L-1}}}\mathcal{B}^{(L,k_{L})}\right)\right)\dots\right),
    \end{equation}
     where 
    \begin{itemize}
        \item $\mathcal{B}^{(0,1)}$ has order $m_{0,1}$ and size $r_{1,1}\times\dots\times r_{1,m_{0,1}}$ and is called root tensor;
        \item if $|\overline{I_{\ell,k}}|\geq 2$ then $\mathcal{B}^{(\ell,k)}$ is a tensor of order $(1+m_{\ell,k})$ with size $r_{\ell,k}\times r_{\ell+1,c_{\ell, k}}\times\dots\times r_{\ell+1,c_{\ell,k}+m_{\ell,k}-1}$, where $c_{\ell, k}$ is the position of the first child of $I_{\ell,k}$ and is called transfer tensor;
        \item if $\overline{I_{\ell,k}}=\{\mu_1, \dots, \mu_s\}$ is a leaf, $\mathcal{B}^{(\ell,k)}$ is a matrix of size $r_{\ell,k}\times (n_{\mu_1}n_{\mu_2}\cdots n_{\mu_s})$ and is called leaf tensor;
        \item if $\overline{I_{\ell,k}}=\{\phantom{4}\}$, then $\mathcal{B}^{(\ell,k)}$ is the identity of proper size and is called dummy tensor.        
    \end{itemize}
    and $\overline{\mathcal{R}}=(r_{\ell,k})_{\overline{I_{\ell,k}}\in\Bar{\mathcal{I}}}$ is a tuple of positive integers.
\end{defi}
The tuple $\mathcal{R}$ obtained from $\overline{\mathcal{R}}$ by removing the integers corresponding to the dummy nodes is called the tree tensor network representation rank.  Figure \ref{fig:toy_example_cubes} illustrates the structure of a TTN factorization with index tree from our toy example. 

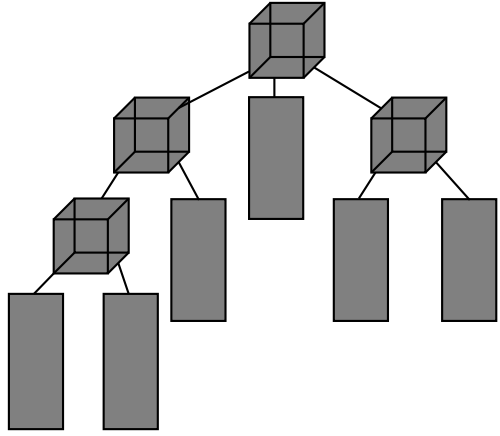
\begin{figure}
    \centering
\begin{tikzpicture}[scale=0.18]

\coordinate (A) at (0,0,0);
\coordinate (B) at (4,0,0);
\coordinate (C) at (4,4,0);
\coordinate (D) at (0,4,0);
\coordinate (E) at (0,0,4);
\coordinate (F) at (4,0,4);
\coordinate (G) at (4,4,4);
\coordinate (H) at (0,4,4);

\fill[gray] (A)--(B)--(C)--(D)--cycle;
\fill[gray] (A)--(B)--(F)--(E)--cycle;
\fill[gray] (A)--(E)--(H)--(D)--cycle;
\draw[thick] (A)--(B)--(C)--(D)--cycle; 
\draw[thick] (E)--(F)--(G)--(H)--cycle; 
\draw[thick] (A)--(E) (B)--(F) (C)--(G) (D)--(H); 

\coordinate (A1) at (-10,-7,0);
\coordinate (B1) at (-6,-7,0);
\coordinate (C1) at (-6,-3,0);
\coordinate (D1) at (-10,-3,0);
\coordinate (E1) at (-10,-7,4);
\coordinate (F1) at (-6,-7,4);
\coordinate (G1) at (-6,-3,4);
\coordinate (H1) at (-10,-3,4);

\fill[gray] (A1)--(B1)--(C1)--(D1)--cycle;
\fill[gray] (A1)--(B1)--(F1)--(E1)--cycle;
\fill[gray] (A1)--(E1)--(H1)--(D1)--cycle;

\draw[thick] (A1)--(B1)--(C1)--(D1)--cycle; 
\draw[thick] (E1)--(F1)--(G1)--(H1)--cycle; 
\draw[thick] (A1)--(E1) (B1)--(F1) (C1)--(G1) (D1)--(H1); 

\coordinate (A2) at (-0.8,-2.2,2);
\coordinate (B2) at (3.2,-2.2,2);
\coordinate (C2) at (3.2,-11.2,2);
\coordinate (D2) at (-0.8,-11.2,2);
\fill[gray] (A2)--(B2)--(C2)--(D2)--cycle;
\draw[thick] (A2)--(B2)--(C2)--(D2)--cycle;

\coordinate (A3) at (9,-7,0);
\coordinate (B3) at (13,-7,0);
\coordinate (C3) at (13,-3,0);
\coordinate (D3) at (9,-3,0);
\coordinate (E3) at (9,-7,4);
\coordinate (F3) at (13,-7,4);
\coordinate (G3) at (13,-3,4);
\coordinate (H3) at (9,-3,4);

\fill[gray] (A3)--(B3)--(C3)--(D3)--cycle;
\fill[gray] (A3)--(B3)--(F3)--(E3)--cycle;
\fill[gray] (A3)--(E3)--(H3)--(D3)--cycle;

\draw[thick] (A3)--(B3)--(C3)--(D3)--cycle; 
\draw[thick] (E3)--(F3)--(G3)--(H3)--cycle; 
\draw[thick] (A3)--(E3) (B3)--(F3) (C3)--(G3) (D3)--(H3); 

\coordinate (A4) at (-16,-16,-4);
\coordinate (B4) at (-12,-16,-4);
\coordinate (C4) at (-12,-12,-4);
\coordinate (D4) at (-16,-12,-4);
\coordinate (E4) at (-16,-16,0);
\coordinate (F4) at (-12,-16,0);
\coordinate (G4) at (-12,-12,0);
\coordinate (H4) at (-16,-12,0);

\fill[gray] (A4)--(B4)--(C4)--(D4)--cycle;
\fill[gray] (A4)--(B4)--(F4)--(E4)--cycle;
\fill[gray] (A4)--(E4)--(H4)--(D4)--cycle;

\draw[thick] (A4)--(B4)--(C4)--(D4)--cycle; 
\draw[thick] (E4)--(F4)--(G4)--(H4)--cycle; 
\draw[thick] (A4)--(E4) (B4)--(F4) (C4)--(G4) (D4)--(H4); 

\coordinate (A5) at (-5,-8.2,6);
\coordinate (B5) at (-1,-8.2,6);
\coordinate (C5) at (-1,-17.2,6);
\coordinate (D5) at (-5,-17.2,6);
\fill[gray] (A5)--(B5)--(C5)--(D5)--cycle;
\draw[thick] (A5)--(B5)--(C5)--(D5)--cycle;

\coordinate (A6) at (-17,-15.2,6);
\coordinate (B6) at (-13,-15.2,6);
\coordinate (C6) at (-13,-25.2,6);
\coordinate (D6) at (-17,-25.2,6);
\fill[gray] (A6)--(B6)--(C6)--(D6)--cycle;
\draw[thick] (A6)--(B6)--(C6)--(D6)--cycle;

\coordinate (A7) at (-10,-15.2,6);
\coordinate (B7) at (-6,-15.2,6);
\coordinate (C7) at (-6,-25.2,6);
\coordinate (D7) at (-10,-25.2,6);
\fill[gray] (A7)--(B7)--(C7)--(D7)--cycle;

\draw[thick] (A7)--(B7)--(C7)--(D7)--cycle;

\coordinate (A8) at (7,-8.2,6);
\coordinate (B8) at (11,-8.2,6);
\coordinate (C8) at (11,-17.2,6);
\coordinate (D8) at (7,-17.2,6);
\fill[gray] (A8)--(B8)--(C8)--(D8)--cycle;
\draw[thick] (A8)--(B8)--(C8)--(D8)--cycle;

\coordinate (A9) at (15,-8.2,6);
\coordinate (B9) at (19,-8.2,6);
\coordinate (C9) at (19,-17.2,6);
\coordinate (D9) at (15,-17.2,6);
\fill[gray] (A9)--(B9)--(C9)--(D9)--cycle;
\draw[thick] (A9)--(B9)--(C9)--(D9)--cycle;

\coordinate (I01) at (-0.8,-0.3,2);
\coordinate (J01) at (-6,-3,2);
\draw[thick] (I01)--(J01);

\coordinate (I02) at (0.3,-1.5,0);
\coordinate (J02) at (0.3,-3,0);
\draw[thick] (I02)--(J02);

\coordinate (I03) at (4,0,2);
\coordinate (J03) at (8.2,-3.8,0);
\draw[thick] (I03)--(J03);

\coordinate (I04) at (-12,-9.3,-2);
\coordinate (J04) at (-14,-12,-4);
\draw[thick] (I04)--(J04);

\coordinate (I05) at (-6,-7,2);
\coordinate (J05) at (-4.5,-9.8,2);
\draw[thick] (I05)--(J05);

\coordinate (I06) at (7,-9.3,-2);
\coordinate (J06) at (5,-12,-4);
\draw[thick] (I06)--(J06);

\coordinate (I07) at (13,-7,2);
\coordinate (J07) at (15.5,-9.8,2);
\draw[thick] (I07)--(J07);

\coordinate (I08) at (-16.7,-16.7,-2);
\coordinate (J08) at (-19,-19.05,-4);
\draw[thick] (I08)--(J08);

\coordinate (I09) at (-12,-16,-2);
\coordinate (J09) at (-12,-19.05,-4);
\draw[thick] (I09)--(J09);

\end{tikzpicture}
    \caption{TTN decomposition of a tensor using the index tree from our toy example.}
    \label{fig:toy_example_cubes}
\end{figure}

Note that the TTN representation in Definition \ref{def: tree tensor network format} reduces to the Tucker format when the index tree is like the one in Figure \ref{fig: MLN_TT_diagrams} (left), while reduces to the TT format when the index tree is like the one in Figure \ref{fig: MLN_TT_diagrams} (right).

In the following, we often opt for the matrix representation of tensors, for which Equation \eqref{eq: tree_tensor_def} becomes
    \begin{equation}\label{eq: tree_tensor_def_matricized}
         {\mathcal{T}}^D =\left(\mathop{\otimes}\limits_{k=1}^{{K}_L}
          B_{L, k}\right)
          \left(\mathop{\otimes}\limits_{k=1}^{{K}_{L-1}}
          B_{L-1, k}\right)
          \dots
          \left(\mathop{\otimes}\limits_{k=1}^{{K}_1}
          B_{1, k}\right)B_{0,1} =\left(\prod_{\ell=L}^1 
         \left(\mathop{\otimes}\limits_{k=1}^{{K}_\ell}
          B_{\ell, k}\right)\right)B_{0,1}.
    \end{equation}
    Throughout this work, the symbol \(\otimes\) represents the Kronecker product. The notation \(B_{\ell,k}\) refers to the matricization of \(\mathcal{B}_{\ell,k}\) over all indices except the first. An exception is made for the root, where \(B_{0,1}\) denotes the vectorization of \(\mathcal{B}_{0,1}\).

\begin{figure}
    \centering   
\begin{tikzpicture}

        \node (C1) at (1.5,1) {$\{1,2,\dots, d\}$};
    \node (B1) at (-0.5,-1) {$\{1\}$};
    \node (B2) at (0.5,-1) {$\{2\}$};
    \node (B3) at (1.5,-1) {$\dots$};
    \node (B4) at (2.5,-1) {$\dots$};
    \node (B5) at (3.5,-1) {$\{d\}$};

    \draw (B1) -- (C1);
    \draw (B2) -- (C1);
    \draw (B3) -- (C1);
    \draw (B4) -- (C1);
    \draw (B5) -- (C1);

     \node (D1) at (6.5,2) {$\{1,2,\dots, d\}$};
    \node (E2) at (6.5,1) {$\{1,2,\dots, d-1\}$};
    \node (F2) at (6.5,0) {$\dots$};
    \node (G2) at (6.5,-1) {$\{1,2\}$};
    \node (H2) at (6.5,-2) {$\{1\}$};

    \draw (D1) -- (E2);
    \draw (E2) -- (F2);
    \draw (F2) -- (G2);
    \draw (G2) -- (H2);
    
\end{tikzpicture}
    \caption{Index tree of Tucker decomposition (left) and Tensor-Train decomposition (right).}
\label{fig: MLN_TT_diagrams}
\end{figure}
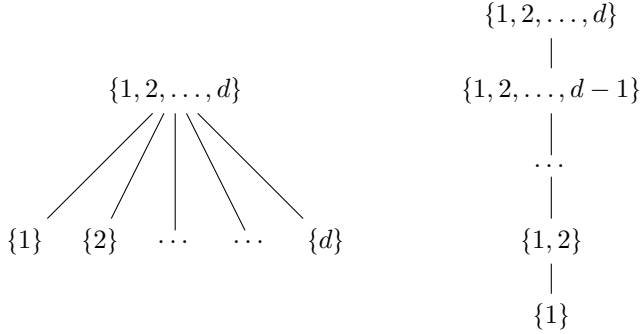
A classic way to visualize a tensor decomposition is through the use of tensor diagrams. In Figure \ref{fig: example of TTN representation } we show the tensor diagram of a 6-dimensional tensor with the index tree in Figure \ref{fig: example of index tree}. In the diagram, we have depicted dummy tensors in yellow. These tensors have no theoretical role in the tensor representation but are useful in practice to lighten the notation. 
\begin{figure}
    \centering
    \begin{tikzpicture}[
        bluenode/.style={rectangle,draw,rounded corners,minimum width=1.5em,minimum height=1.5em, fill=blue!20},
        yellownode/.style={rectangle,draw,rounded corners,minimum width=1.5em,minimum height=1.5em, fill=yellow!20},
        level distance=1.5cm,
        level 1/.style={sibling distance=4cm},
        level 2/.style={sibling distance=2cm},
        level 3/.style={sibling distance=1.5cm}
    ]
    
    \node[bluenode] (D1) at (6.5,2) {$\mathcal{B}_{0,1}$};
    
    \node[bluenode] (E1) at (4.8,0.9) {$\mathcal{B}_{1,1}$};
    \node[bluenode] (E2) at (6.5,0.9) {$\mathcal{B}_{1,2}$};
    \node[bluenode] (E3) at (8.2,0.9) {$\mathcal{B}_{1,3}$};
    
    \node[bluenode] (F1) at (3.7,-0.2) {$\mathcal{B}_{2,1}$};
    \node[bluenode] (F2) at (5.4,-0.2) {$\mathcal{B}_{2,2}$};
    \node[yellownode] (F22) at (6.5, -0.2) {$\mathcal{B}_{2,3}$};
    \node[bluenode] (F3) at (7.6,-0.2) {$\mathcal{B}_{2,4}$};
    \node[bluenode] (F4) at (8.8,-0.2) {$\mathcal{B}_{2,5}$};
    \node[bluenode] (G1) at (3.0,-1.3) {$\mathcal{B}_{3,1}$};
    \node[bluenode] (G2) at (4.2,-1.3) {$\mathcal{B}_{3,2}$};
    \node[yellownode] (G3) at (5.4,-1.3) {$\mathcal{B}_{3,3}$};
    \node[yellownode] (G4) at (6.5, -1.3) {$\mathcal{B}_{3,4}$};
    \node[yellownode] (G5) at (7.6,-1.3) {$\mathcal{B}_{3,5}$};
    \node[yellownode] (G6) at (8.8,-1.3) {$\mathcal{B}_{3,6}$};
    \node (L1) at (3.0,-2.3) {};
    \node (L2) at (4.2,-2.3) {};
    \node (L3) at (5.4,-2.3) {};
    \node (L4) at (6.5,-2.3) {};
    \node (L5) at (7.6,-2.3) {};
    \node (L6) at (8.8,-2.3) {};
    
    \draw (D1) -- (E1) node[midway,above] {$r_{1,1}$};
    \draw (D1) -- (E2) node[midway, right] {$r_{1,2}$};
    \draw (D1) -- (E3) node[midway, above] {$r_{1,3}$};
    
    \draw (E1) -- (F1) node[midway,left] {$r_{2,1}$};
    \draw (E1) -- (F2) node[midway,right] {$r_{2,2}$};
    \draw (E2) -- (F22) node[midway,right] {$r_{2,3}$};
    \draw (E3) -- (F3) node[midway,left] {$r_{2,4}$};
    \draw (E3) -- (F4) node[midway,right] {$r_{2,5}$};
    
    \draw (F1) -- (G1) node[midway,left] {$r_{3,1}$};
    \draw (F1) -- (G2) node[midway,right] {$r_{3,2}$};
    \draw (F2) -- (G3) node[midway,right] {$r_{3,3}$};
   \draw (F22) -- (G4) node[midway,right] {$r_{3,4}$};
    \draw (F3) -- (G5) node[midway,right] {$r_{3,5}$};
    \draw (F4) -- (G6) node[midway,right] {$r_{3,6}$};

    \draw (G1) -- (L1) node[midway,left] {$n_1$};
    \draw (G2) -- (L2) node[midway,right] {$n_2$};
    \draw (G3) -- (L3) node[midway,right] {$n_3$};
    \draw (G4) -- (L4) node[midway,right] {$n_4$};
    \draw (G5) -- (L5) node[midway,right] {$n_5$};
    \draw (G6) -- (L6) node[midway,right] {$n_6$};
    
    \end{tikzpicture}
    \caption{Tensor diagram of a TTN representation with the index tree in Figure \ref{fig: example of index tree}.}
    \label{fig: example of TTN representation }
\end{figure}
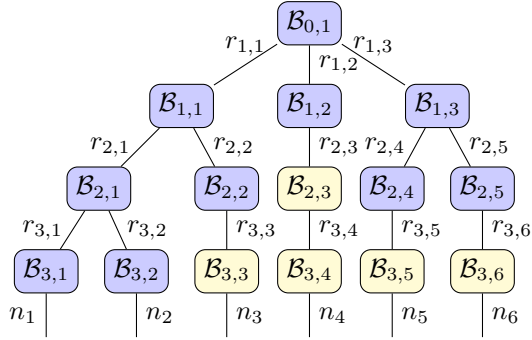

In the following, when we say that a tensor has TTN-rank $\mathcal{R}$, we mean that there exists a representation of the tensor with $\mathcal{R}$ as TTN-rank (the associated index tree will be clear from the context). Given two tensors $\mathcal{A}$ of TTN-rank $\mathcal{R}_\mathcal{A}$ and $\mathcal{B}$ of TTN-rank $\mathcal{R}_\mathcal{B}$, we say that $\mathcal{R}_\mathcal{A} < \mathcal{R}_\mathcal{B}$ if the TTN-rank of $\mathcal A$ is component-wise smaller than the one of $\mathcal B$. 


\section{Tree tensor network approximation}\label{sec:TTNN}
In this section, we define our streamable and single pass algorithm that, given an index tree $\mathcal{I}$ and a tuple $\mathcal{R}=(r_{\ell,k})_{I_{\ell,k}\in\mathcal{I}}$, provides a low-rank tree tensor network approximation of a tensor $\mathcal{T}$ of size $n_1 \times\dots\times n_d$ with TTN representation rank $\mathcal{R}$. 

Recently, various GN-based methods have been formulated to calculate low-rank approximations of tensors in different formats: Tucker \cite{Another_Tucker,bucci2024multilinear, bucci2024sequential}, tensor train \cite{rounding_TT, kressner2023streaming}, and tensor ring \cite{tensorringGN}.
The structure of these algorithms is largely similar. In essence, they can all be reduced to a sequence of GN projections across different tensor modes. With GN projection we mean the following: consider a tensor $\mathcal{T}\in \mathbb{R}^{n_1\times \dots \times n_d}$, a subset of indices $I\subset D$ and two DRMs $X_I\in \mathbb{R}^{n_I\times r_I}$ and $Y_I\in \mathbb{R}^{n_I\times (r_I+ p_I)}$, the projection is defined by
\begin{equation}
    P_{I} = \mathcal{T}^I X_I (Y_I^T \mathcal{T}^I X_I)^\dagger Y_I^T
\end{equation}
The multilinear Nystr\"om (MLN) \cite{bucci2024multilinear} for instance approximate the tensor $\mathcal{T}$ in Tucker format in the following way
\begin{equation}\label{eq: MLN}
   \hat{\mathcal{T}}_{MLN} = \mathcal{T} \times_{\mu=1}^d P_{\{\mu\}}. 
\end{equation}
Similarly, setting $P_{1:k}:=P_{\{1, \dots, k\}}$ the streaming tensor train approximation method (STTA) \cite{kressner2023streaming} approximates the tensor in TT format as
\begin{equation} \label{eq: STTA}
    \hat{\mathcal{T}}^{\{1,\dots, d-1\}}_{STTA}=(P_1 \otimes I)\dots (P_{1:d-2}\otimes I)P_{1:d-1}\mathcal{T}^{\{1,\dots, d-1\}}.
\end{equation}

\noindent  
Our method extends these approaches to a general TTN representation. First, we need to adapt the notation of DRMs and projectors. 
Given an index tree $\mathcal{I}$, a tuple of target ranks $\mathcal{R} = \{r_{\ell, k}\}$, and a tuple of oversamples $\mathcal{P} = \{p_{\ell, k}\}$, for each $I_{\ell, k}\in \Bar{\mathcal{I}}$, except for the root, we define the following DRMs:
\begin{equation}\label{sketch}
\begin{split}
    & X_{\ell,k} \text{ is a DRM of size: } n_{D\backslash I_{\ell,k}}
    \times r_{\ell,k},\\
    & Y_{\ell,k}\text{ is a DRM of size: }n_{I_{\ell, k}}
    \times(r_{\ell,k} + p_{\ell,k}),\\
\end{split}   
\end{equation}
and the following oblique projectors
\begin{equation}\label{projectors}
    P_{\ell,k}:=\mathcal{T}^{\ell,k}X_{\ell,k}(Y_{\ell,k}^T\mathcal{T}^{\ell,k}X_{\ell,k})^\dagger Y_{\ell,k}^T.
\end{equation}
We use the convention that $X_{\ell,k}$, $Y_{\ell,k}$ and $P_{\ell,k}$ are all identities if $I_{\ell,k}$ is a dummy node. 
With the latter clarification, the treatment of dummy nodes will be clear. Therefore, in what follows, we will assume that the index tree is equal to the extended index tree (there are no dummy nodes), allowing us to simplify the notation and avoid the constant use of the overline symbol.

The \textbf{tree tensor network Nyström} (TTNN) approximant $\widehat{\mathcal{T}}_{TTNN}$ of $\mathcal{T}$ with index tree $\mathcal{I}$ can be computed by performing the following sequence of projections
\begin{align*} \label{eq:TTNN_projections}
    &\mathcal{T}^{(1)} = \mathcal{T}\times_{1,1}P_{1,1}\times_{1,2} P_{1,2}\dots \times P_{1,K_1}, \\
    &\mathcal{T}^{(2)} = \mathcal{T}^{(1)}\times_{2,1}P_{2,1}\times_{2,2} P_{2,2}\dots \times P_{2,K_2},\\
    & \,\,\,\qquad\vdots\\
    & \mathcal{T}^{(L)}=\mathcal{T}^{(L-1)}\times_{L,1}P_{L,1}\times_{L,2} P_{L,2}\dots \times P_{L,K_L},
\end{align*}
and by setting $\widehat{\mathcal{T}}_{TTNN} := \mathcal{T}^{(L)}$.
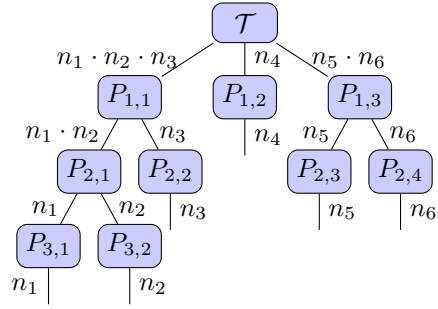
\begin{figure}
    \centering
    \begin{tikzpicture}[scale=0.9,
        bluenode/.style={rectangle,draw,rounded corners,minimum width=1.5em,minimum height=1.5em, fill=blue!20},
        level distance=1.5cm,
        level 1/.style={sibling distance=4cm},
        level 2/.style={sibling distance=2cm},
        level 3/.style={sibling distance=1.5cm}
    ]
    
    \node[bluenode] (D1) at (6.5,2) {$\phantom{0}\mathcal{T}{\phantom{0}}$};
    
    \node[bluenode] (E1) at (4.8,0.9) {$P_{1,1}$};
    \node[bluenode] (E2) at (6.5,0.9) {$P_{1,2}$};
    \node[bluenode] (E3) at (8.2,0.9) {$P_{1,3}$};
    
    \node[bluenode] (F1) at (4.2,-0.2) {$P_{2,1}$};
    \node[bluenode] (F2) at (5.4,-0.2) {$P_{2,2}$};
    \node[bluenode] (F3) at (7.6,-0.2) {$P_{2,3}$};
    \node[bluenode] (F4) at (8.8,-0.2) {$P_{2,4}$};
    \node[bluenode] (G1) at (3.6,-1.3) {$P_{3,1}$};
    \node[bluenode] (G2) at (4.8,-1.3) {$P_{3,2}$};

    \node (L1) at (3.6,-2.3) {};
    \node (L2) at (4.8,-2.3) {};
    \node (L3) at (5.4,-1.2) {};
    \node (L4) at (6.5,-0.1) {};
    \node (L5) at (7.6,-1.2) {};
    \node (L6) at (8.8,-1.2) {};
    
    \draw (D1) -- (E1) node[midway,left] {$n_1\cdot n_2 \cdot n_3$};

    \draw (D1) -- (E2) node[midway, right] {$n_4$};
    \draw (D1) -- (E3) node[midway, right] {$n_5 \cdot n_6$};
    
    \draw (E1) -- (F1) node[midway,left] {$n_1 \cdot n_2$};
    \draw (E1) -- (F2) node[midway,right] {$n_3$};
    \draw (E3) -- (F3) node[midway,left] {$n_5$};
    \draw (E3) -- (F4) node[midway,right] {$n_6$};
    
    \draw (F1) -- (G1) node[midway,left] {$n_1$};
    \draw (F1) -- (G2) node[midway,right] {$n_2$};

    \draw (G1) -- (L1) node[midway,left] {$n_1$};
    \draw (G2) -- (L2) node[midway,right] {$n_2$};
    \draw (F2) -- (L3) node[midway,right] {$n_3$};
    \draw (E2) -- (L4) node[midway,right] {$n_4$};
    \draw (F3) -- (L5) node[midway,right] {$n_5$};
    \draw (F4) -- (L6) node[midway,right] {$n_6$};    
    
    \end{tikzpicture}
    \caption{TTNN approximation of a six-mode tensor with the index tree in figure \ref{fig: example of index tree}.}
    \label{fig: 6-dim with projectors}
\end{figure}
Otherwise, using matricizations, we can express the formula in a more compact way as
\begin{equation}\label{TTNN approximant projector}
    \widehat{\mathcal{T}}^{D}_{TTNN} = \left(\prod_{\ell=L}^1 \left(\mathop{\otimes}\limits_{k=1}^{{K_\ell}} P_{\ell, k}\right)\right)\mathcal{T}^D.
\end{equation}
Note that this approximation retrieves the MLN approximation if the index tree is like the one in Figure \ref{fig: MLN_TT_diagrams} (left) and retrieves the STTA approximation if the index tree is like the one in Figure \ref{fig: MLN_TT_diagrams} (right).
\par
However, computing the approximant as described above is computationally demanding and does not yield a compressed representation in TTN format. Nevertheless, we chose to introduce the approximant in this way as it provides a more intuitive understanding, a clear geometric interpretation, and will serve for our analysis. We now describe the process for obtaining the TTNN approximant in TTN format with ranks $\mathcal{R}$. For reference, the relevant parameters are described in Definition \ref{def: tree tensor network format}.
The approach mirrors that of Algorithm \ref{alg:GN}: first, we compute all the sketches, then we perform the QR factorizations, and finally we construct the transfer tensors. \phantom{a}\\
\phantom{a}\\
\phantom{a}$\bullet$ \textbf{Sketch phase}\\ 
\phantom{a}\\
  \phantom{a,.,a}$\circ$ $\Omega_{\ell, k} = Y^T_{\ell,k}\mathcal{T}^{\ell,k} X_{\ell,k}$ if $I_{\ell, k}$ is not the root.
    \phantom{a}\\
    \phantom{a}\\
    \vspace{5mm}
  \phantom{abc}$\circ$ $\Psi_{\ell, k} = \begin{cases} 
        (Y_{1,1}^T\otimes\dots\otimes Y_{1,m_{0,1}}^T) \mathcal{T}^D &\text{if }I_{\ell,k}\text{ is the root,}\\
        (Y_{\ell+1,c_{\ell,k}}^T\otimes\dots \otimes Y_{\ell+1,c_{\ell,k}+m_{\ell,k}-1}^T )\mathcal{T}^{\ell,k} X_{\ell,k}&\text{if }I_{\ell,k}\text{ is not root or leaf,}\\
         \mathcal{T}^{\ell,k} X_{\ell,k}& \text{if }I_{\ell,k}\text{ is a leaf.}
    \end{cases}$\\
  
   \vspace{-3mm}\phantom{a}\hspace{-6mm}$\bullet$ \textbf{Recovery phase}\\   
   \phantom{a}\\
        \phantom{abc}$\circ$ $[Q_{\ell, k}, R_{\ell, k}] =\text{qr}(\Omega_{\ell,k})$,\\
        \phantom{a}\\
        \phantom{abc}$\circ$ $B_{\ell, k} =\begin{cases}
            Q_{\ell, k}^T \Psi_{\ell, k} & \text{if }I_{\ell,k}\text{ is the root,} \\
            Q_{\ell, k}^T\Psi_{\ell, k} R_{\ell, k}^\dagger & \text{if }I_{\ell,k}\text{ is not root or leaf,}\\
            \Psi_{\ell, k}R_{\ell, k}^{\dagger} & \text{if }I_{\ell,k}\text{ is a leaf.}
        \end{cases}$\\
        \\
\noindent See Algorithm \ref{alg:TTNN} for the pseudocode of the method.

Because the sketch phase is linear in $\mathcal{T}$ and the recovery phase only involves small objects, the algorithm is inherently streamable and one-pass. To illustrate this, consider the scenario where we aim to compute the TTNN approximation of a tensor expressed as a linear combination of other tensors, $\mathcal{T} = \lambda_1 \mathcal{H}_1 + \dots + \lambda_s \mathcal{H}_s$ and assume that each tensor $\mathcal{H}_s$ can only be accessed once and must be discarded before the next tensor is processed.
Using the same sketchings $X_{\ell, k}$ and $Y_{\ell, k}$ for each $\mathcal{H}_i$, compute the small matrices $\Omega^{(i)}_{\ell, k}$ and $\Psi^{(i)}_{\ell, k}$  and then form the linear combinations $\Omega_{\ell, k} = \sum_{i=1}^s \lambda_i \Omega_{\ell, k}^{(i)}$ and $\Psi_{\ell, k} = \sum_{i=1}^s \lambda_i \Psi_{\ell, k}^{(i)}$, which are exactly the sketchings of $\mathcal{T}$.
Once the $\Psi_{\ell, k}$ and $\Omega_{\ell, k}$ are computed proceed as described in the recovery phase to obtain the $B_{\ell, k}$.

A possible cause for concern is the stability of the algorithm as it involves the matrices $R_{\ell, k}^\dagger$. However, it is possible to never form these pseudoinverses explicitly and to equivalently solve the following least square problems
\begin{equation}
    \min_{X} \| \Psi_{\ell, k}- X R_{\ell, k}\|_F.
\end{equation}
Nonetheless, care needs to be taken when $R_{\ell, k}$ is ill-conditioned. An alternative is to solve the stabilized least square
\begin{equation}
    \| \Psi_{\ell, k}- X (R_{\ell, k})_\epsilon\|_F,
\end{equation}
where $\epsilon = 10 u\|R_{\ell, k}\|_2$ and $u$ denotes the machine precision, see \cite{nakatsukasa2020fast} for further details.
\begin{algorithm}
\caption{Tree Tensor Network Nystr\"om (TTNN)}\label{alg:TTNN}
\begin{algorithmic}[1]
    \Require  $\mathcal{T}\in \mathbb{R}^{n_1\times\dots \times n_d}$, index tree $\mathcal{I}$, TTN rank $\mathcal{R}$, TTN oversamples $\mathcal{P}$.
    \Ensure $\{\mathcal{B}_{\ell, k}\}_{(\ell, k)\in \mathcal{I}}$, TTN representation of $\mathcal{T}$ of rank $\mathcal{R}$.
    \For{$\ell=1, \dots, L$}
    \For{$k=1, \dots, K_{\ell}$}
    \State Draw random matrix $X_{\ell, k}\in     \mathbb{R}^{n_{D/I_{\ell,k}}
    \times r_{\ell,k}}$,
    \State Draw random matrix $Y_{\ell, k}\in \mathbb{R}^{n_{I_{\ell, k}}
    \times(r_{\ell,k} + p_{\ell,k})}$;
    \State \mbox{Compute an economy-size QR
factorization $Y_{\ell, k}^T\mathcal{T}^{\ell, k} X_{\ell, k} = Z_{\ell,k} R_{\ell,k}$;}
    \EndFor
    \EndFor 
    \State Compute \mbox{$B_{0,1}=(Z_{1,1}^T Y_{1,1}^T\otimes \dots \otimes Z_{1, m_{0,1}}^T Y_{1,m_{0,1}}^T)\mathcal{T}^{0, 1}$.}
    \For{$\ell=2, \dots, L-1$}
    \For{$k=1, \dots, K_{\ell}$}
    \State Set $W_{\ell, k} =Z_{c(\ell, k)}^T Y_{c(\ell, k)}^T\otimes \dots \otimes Z_{c(\ell, k)+m_{\ell, k}-1}^T Y_{c(\ell, k)+m_{\ell, k}-1}^T$
    \State Compute \mbox{$B_{\ell, k}=W_{\ell, k}\mathcal{T}^{\ell, k}X_{\ell, k}R_{\ell, k}^{\dagger}$}
    \EndFor
    \EndFor
    \For{$k=1, \dots, K_L$}
    \State $B_{L,k}= \mathcal{T}^{L,k} X_{L,k} R_{L, k}^\dagger$.
    \EndFor
    \State Return $\{\mathcal{B}_{\ell, k}\}_{(\ell, k)\in \mathcal{I}}$.
\end{algorithmic}
\end{algorithm}

\section{Sequential tree tensor network approximation}\label{sec:STTNN}

We have seen how to construct a TTN approximation of a tensor using the TTNN algorithm. 
In particular, the algorithm requires multiplying the entire tensor $\mathcal{T}$, properly unfolded, by two random matrices, one on the left and one on the right, for each node in the index tree.  
However, this procedure does not fully exploit the capabilities of the generalized Nyström algorithm which allows for sequentially reducing the size of the tensor involved in the sketching procedure, still requiring only one pass over the data and preserving its streamability \cite{bucci2024sequential}. 

Henceforth, in this section, we present the sequential tree tensor network Nyström (STTNN) approximant, a sequential variant of the TTNN method that is still streamable and single-pass.
To avoid delving into technical details, we will not provide a formal description or pseudocode for the STTNN algorithm. Instead, we offer a step-by-step overview.

In TTNN, the projectors $P_{\ell, k}$ take the form \eqref{projectors}.
In the STTNN algorithm, we just replace these projections with cheaper ones, by replacing the matrices $\mathcal{T}^{\ell, k}$ with progressively smaller ones as we compute new projections and new tensor contractions are introduced. This process is similar to what is done in the sequential multilinear Nystr\"om method \cite{bucci2024sequential}, but with notable differences due to the hierarchical structure of the problem.

The key idea is that whenever two sets of indices are disjoint, we can exploit previously computed contractions to reduce the computation costs. Let us clarify this point. Let \(I_{\ell_1, k_1}\) and \(I_{\ell_2, k_2}\) be two disjoint sets of indices, and consider the following possible approximations of \(\mathcal{T}^{\ell_2, k_2}\)
\begin{equation}\label{eq: TTNN example proj}
\mathcal{T}^{\ell_2, k_2} \approx \mathcal{T}^{\ell_2, k_2} X_{\ell_2, k_2} \left( Y_{\ell_2, k_2}^T \mathcal{T}^{\ell_2, k_2} X_{\ell_2, k_2} \right)^\dagger Y_{\ell_2, k_2}^T \mathcal{T}^{\ell_2, k_2}
\end{equation}
and
\begin{equation} \label{eq: STTN example proj}
\mathcal{T}^{\ell_2, k_2} \approx (\mathcal{T} \times_{\ell_1, k_1} Y_{\ell_1, k_1}^T)^{\ell_2, k_2} X_{\ell_2, k_2} \left( Y_{\ell_2, k_2}^T (\mathcal{T} \times_{\ell_1, k_1} Y_{\ell_1,k_1}^T)^{\ell_2, k_2} X_{\ell_2, k_2} \right)^\dagger Y_{\ell_2, k_2}^T \mathcal{T}^{\ell_2, k_2},
\end{equation}
where with a slight abuse of notation in \eqref{eq: STTN example proj}, we used  $\mathcal{T} \times_{\ell_1, k_1} Y_{\ell_1, k_1}^T$ to denote the contraction of $\mathcal{T}$ along $I_{\ell_1, k_1}$ with the first dimension of $Y_{\ell_1, k_1}$ and with $X_{\ell_2,k_2}$ a suitably sized DRM (with fewer rows than in TTNN).
Both expressions are generalized Nyström approximations, differing only in the sketching matrices used, and should therefore yield comparable results.
However, since the second approximation leverages the precomputed contracted tensor $\mathcal{T} \times_{\ell_1, k_1} Y_{\ell_1, k_1}^T$ instead of the full tensor $\mathcal{T}$, constructing the projection
\begin{equation}\label{eq: 1sketch}
P_{\ell_2, k_2} := (\mathcal{T} \times_{\ell_1, k_1} Y_{\ell_1, k_1}^T)^{\ell_2, k_2} X_{\ell_2, k_2} \left( Y_{\ell_2, k_2}^T (\mathcal{T} \times_{\ell_1, k_1} Y_{\ell_1, k_1}^T)^{\ell_2, k_2} X_{\ell_2, k_2} \right)^\dagger Y_{\ell_2, k_2}^T
\end{equation}
is computationally cheaper. All these observations can be extended to more sets of disjoint indices. For instance, if $I_{\ell_3, k_3}$ is disjoint from both $I_{\ell_1, k_1}$ and $I_{\ell_2, k_2}$, the projection $P_{\ell_3, k_3}$ may involve the tensor $\mathcal{T} \times_{\ell_1, k_1} Y_{\ell_1, k_1}^T\times_{\ell_2, k_2} Y_{\ell_2, k_2}^T$.

So in the STTNN algorithm, we replace the TTNN projections with projections involving tensors previously contracted. This introduces the sequential structure to the process. Denoting with $S_{\ell, k}=\{I_{\ell_1, k_1}, \dots, I_{\ell_s, k_s}\}$ the set of indices disjoint from $I_{\ell, k}$ for which we have already computed the contractions and with $\mathcal{T}_{S_{\ell, k}}:=\mathcal{T}\times_{(\ell_i, k_i)\in S_{\ell, k}} Y_{\ell_i, k_i}^T$ the tensor contracted along the set of indices in $S_{\ell, k}$, the projections in the STTNN algorithm would have the following form
\begin{equation}\label{eq: general_seq_projection}
    P_{\ell, k} = \mathcal{T}_{S_{\ell,k}}^{\ell, k}X_{\ell, k}(Y^T_{\ell, k}\mathcal{T}_{S_{\ell,k}}^{\ell, k} X_{\ell, k})^\dagger Y_{\ell, k}^T.
\end{equation}
Although the individual projections will differ from those used in TTNN, the hierarchical structure remains identical, and the following formula for the STTNN approximant applies
\begin{equation}\label{STTNN approximant projector}
    \widehat{\mathcal{T}}^{D}_{STTNN} = \left(\prod_{\ell=L}^1 \left(\mathop{\otimes}\limits_{k=1}^{{K_\ell}} P_{\ell, k}\right)\right)\mathcal{T}^D.
\end{equation}

Thus, summarizing, the STTNN algorithm selects an ordering of the nodes and computes the projections $P_{\ell, k}$, taking advantage of all previously computed contractions with indices disjoint from $I_{\ell, k}$. It is evident that the ordering of the nodes is fundamental to the efficiency of the STTNN algorithm. However, determining the optimal ordering presents a combinatorial challenge and may vary depending on the specific problem at hand \cite{stoian2024optimal}. In this context, we propose a reasonable and intuitive ordering. In our approach, we follow a lexicographic ordering of the nodes. We begin with the first level, starting from \((1,1)\), then proceed to \((1,2)\), and continue in this manner until the last node of the first level. After completing the first level, we move on to the second level, starting with \((2,1)\), followed by \((2,2)\), and so on. This process is repeated until we reach the final node of the last level. For a visual representation of which contractions can be leveraged for a given projection, see Figure \ref{fig: STTNN_indices_ordering}. 

\begin{figure}[ht!]
    \centering
    \begin{tikzpicture}[scale=0.7]
    \node (I01) at (0, 0) {$I_{0,1}$};
    \node (I11) at (-2, -1.3) {$I_{1,1}$};
    \node (I12) at (0, -1.3) {$I_{1,2}$};
    \node (I13) at (2, -1.3) {$I_{1,3}$};
    \node (I21) at (-3, -2.6) {$I_{2,1}$};
    \node (I22) at (-1, -2.6) {$I_{2,2}$};
    \node (I23) at (1, -2.6) {$I_{2,3}$};
    \node (I24) at (3, -2.6) {$I_{2,4}$};
    \node (I31) at (-4, -3.9) {$I_{3,1}$};
    \node (I32) at (-2, -3.9) {$I_{3,2}$};

    \draw[-] (I01) -- (I11);
    \draw[-] (I01) -- (I12);
    \draw[-] (I01) -- (I13);
    \draw[-] (I11) -- (I21);
    \draw[-] (I11) -- (I22);
    \draw[-] (I13) -- (I23);
    \draw[-] (I13) -- (I24);
    \draw[-] (I21) -- (I31);
    \draw[-] (I21) -- (I32);

    \draw[thick] (I11) circle (0.45cm);
\end{tikzpicture}~\quad~\begin{tikzpicture}[scale=0.7]
    \node (I01) at (0, 0) {$I_{0,1}$};
    \node (I11) at (-2, -1.3) {$I_{1,1}$};
    \node (I12) at (0, -1.3) {$I_{1,2}$};
    \node (I13) at (2, -1.3) {$I_{1,3}$};
    \node (I21) at (-3, -2.6) {$I_{2,1}$};
    \node (I22) at (-1, -2.6) {$I_{2,2}$};
    \node (I23) at (1, -2.6) {$I_{2,3}$};
    \node (I24) at (3, -2.6) {$I_{2,4}$};
    \node (I31) at (-4, -3.9) {$I_{3,1}$};
    \node (I32) at (-2, -3.9) {$I_{3,2}$};

    \draw[-] (I01) -- (I11);
    \draw[-] (I01) -- (I12);
    \draw[-] (I01) -- (I13);
    \draw[-] (I11) -- (I21);
    \draw[-] (I11) -- (I22);
    \draw[-] (I13) -- (I23);
    \draw[-] (I13) -- (I24);
    \draw[-] (I21) -- (I31);
    \draw[-] (I21) -- (I32);

    \draw[thick] (I12) circle (0.45cm);
     \draw[thick] (I11) ++(-0.5,-0.3) rectangle ++(1,0.7);
\end{tikzpicture}\\
    \begin{tikzpicture}[scale=0.7]
    \node (I01) at (0, 0) {$I_{0,1}$};
    \node (I11) at (-2, -1.3) {$I_{1,1}$};
    \node (I12) at (0, -1.3) {$I_{1,2}$};
    \node (I13) at (2, -1.3) {$I_{1,3}$};
    \node (I21) at (-3, -2.6) {$I_{2,1}$};
    \node (I22) at (-1, -2.6) {$I_{2,2}$};
    \node (I23) at (1, -2.6) {$I_{2,3}$};
    \node (I24) at (3, -2.6) {$I_{2,4}$};
    \node (I31) at (-4, -3.9) {$I_{3,1}$};
    \node (I32) at (-2, -3.9) {$I_{3,2}$};

    \draw[-] (I01) -- (I11);
    \draw[-] (I01) -- (I12);
    \draw[-] (I01) -- (I13);
    \draw[-] (I11) -- (I21);
    \draw[-] (I11) -- (I22);
    \draw[-] (I13) -- (I23);
    \draw[-] (I13) -- (I24);
    \draw[-] (I21) -- (I31);
    \draw[-] (I21) -- (I32);

    \draw[thick] (I13) circle (0.45cm);
    \draw[thick] (I11) ++(-0.5,-0.3) rectangle ++(1,0.7);
    \draw[thick] (I12) ++(-0.5,-0.3) rectangle ++(1,0.7);
\end{tikzpicture}~\quad~\begin{tikzpicture}[scale=0.7]
    \node (I01) at (0, 0) {$I_{0,1}$};
    \node (I11) at (-2, -1.3) {$I_{1,1}$};
    \node (I12) at (0, -1.3) {$I_{1,2}$};
    \node (I13) at (2, -1.3) {$I_{1,3}$};
    \node (I21) at (-3, -2.6) {$I_{2,1}$};
    \node (I22) at (-1, -2.6) {$I_{2,2}$};
    \node (I23) at (1, -2.6) {$I_{2,3}$};
    \node (I24) at (3, -2.6) {$I_{2,4}$};
    \node (I31) at (-4, -3.9) {$I_{3,1}$};
    \node (I32) at (-2, -3.9) {$I_{3,2}$};

    \draw[-] (I01) -- (I11);
    \draw[-] (I01) -- (I12);
    \draw[-] (I01) -- (I13);
    \draw[-] (I11) -- (I21);
    \draw[-] (I11) -- (I22);
    \draw[-] (I13) -- (I23);
    \draw[-] (I13) -- (I24);
    \draw[-] (I21) -- (I31);
    \draw[-] (I21) -- (I32);

     \draw[thick] (I21) circle (0.45cm);
\end{tikzpicture}\\
    \begin{tikzpicture}[scale=0.7]
    \node (I01) at (0, 0) {$I_{0,1}$};
    \node (I11) at (-2, -1.3) {$I_{1,1}$};
    \node (I12) at (0, -1.3) {$I_{1,2}$};
    \node (I13) at (2, -1.3) {$I_{1,3}$};
    \node (I21) at (-3, -2.6) {$I_{2,1}$};
    \node (I22) at (-1, -2.6) {$I_{2,2}$};
    \node (I23) at (1, -2.6) {$I_{2,3}$};
    \node (I24) at (3, -2.6) {$I_{2,4}$};
    \node (I31) at (-4, -3.9) {$I_{3,1}$};
    \node (I32) at (-2, -3.9) {$I_{3,2}$};

    \draw[-] (I01) -- (I11);
    \draw[-] (I01) -- (I12);
    \draw[-] (I01) -- (I13);
    \draw[-] (I11) -- (I21);
    \draw[-] (I11) -- (I22);
    \draw[-] (I13) -- (I23);
    \draw[-] (I13) -- (I24);
    \draw[-] (I21) -- (I31);
    \draw[-] (I21) -- (I32);

    \draw[thick] (I22) circle (0.45cm);
     \draw[thick] (I21) ++(-0.5,-0.3) rectangle ++(1,0.7);
\end{tikzpicture}~\quad~\begin{tikzpicture}[scale=0.7]
    \node (I01) at (0, 0) {$I_{0,1}$};
    \node (I11) at (-2, -1.3) {$I_{1,1}$};
    \node (I12) at (0, -1.3) {$I_{1,2}$};
    \node (I13) at (2, -1.3) {$I_{1,3}$};
    \node (I21) at (-3, -2.6) {$I_{2,1}$};
    \node (I22) at (-1, -2.6) {$I_{2,2}$};
    \node (I23) at (1, -2.6) {$I_{2,3}$};
    \node (I24) at (3, -2.6) {$I_{2,4}$};
    \node (I31) at (-4, -3.9) {$I_{3,1}$};
    \node (I32) at (-2, -3.9) {$I_{3,2}$};

    \draw[-] (I01) -- (I11);
    \draw[-] (I01) -- (I12);
    \draw[-] (I01) -- (I13);
    \draw[-] (I11) -- (I21);
    \draw[-] (I11) -- (I22);
    \draw[-] (I13) -- (I23);
    \draw[-] (I13) -- (I24);
    \draw[-] (I21) -- (I31);
    \draw[-] (I21) -- (I32);

        \draw[thick] (I23) circle (0.45cm);
     \draw[thick] (I11) ++(-0.5,-0.3) rectangle ++(1,0.7);
     \draw[thick] (I12) ++(-0.5,-0.3) rectangle ++(1,0.7);
\end{tikzpicture}\\
\begin{tikzpicture}[scale=0.7]
    \node (I01) at (0, 0) {$I_{0,1}$};
    \node (I11) at (-2, -1.3) {$I_{1,1}$};
    \node (I12) at (0, -1.3) {$I_{1,2}$};
    \node (I13) at (2, -1.3) {$I_{1,3}$};
    \node (I21) at (-3, -2.6) {$I_{2,1}$};
    \node (I22) at (-1, -2.6) {$I_{2,2}$};
    \node (I23) at (1, -2.6) {$I_{2,3}$};
    \node (I24) at (3, -2.6) {$I_{2,4}$};
    \node (I31) at (-4, -3.9) {$I_{3,1}$};
    \node (I32) at (-2, -3.9) {$I_{3,2}$};

    \draw[-] (I01) -- (I11);
    \draw[-] (I01) -- (I12);
    \draw[-] (I01) -- (I13);
    \draw[-] (I11) -- (I21);
    \draw[-] (I11) -- (I22);
    \draw[-] (I13) -- (I23);
    \draw[-] (I13) -- (I24);
    \draw[-] (I21) -- (I31);
    \draw[-] (I21) -- (I32);

        \draw[thick] (I24) circle (0.45cm);
    \draw[thick] (I11) ++(-0.5,-0.3) rectangle ++(1,0.7);
    \draw[thick] (I12) ++(-0.5,-0.3) rectangle ++(1,0.7);
    \draw[thick] (I23) ++(-0.5,-0.3) rectangle ++(1,0.7);
\end{tikzpicture}~\quad~\begin{tikzpicture}[scale=0.7]
    \node (I01) at (0, 0) {$I_{0,1}$};
    \node (I11) at (-2, -1.3) {$I_{1,1}$};
    \node (I12) at (0, -1.3) {$I_{1,2}$};
    \node (I13) at (2, -1.3) {$I_{1,3}$};
    \node (I21) at (-3, -2.6) {$I_{2,1}$};
    \node (I22) at (-1, -2.6) {$I_{2,2}$};
    \node (I23) at (1, -2.6) {$I_{2,3}$};
    \node (I24) at (3, -2.6) {$I_{2,4}$};
    \node (I31) at (-4, -3.9) {$I_{3,1}$};
    \node (I32) at (-2, -3.9) {$I_{3,2}$};

    \draw[-] (I01) -- (I11);
    \draw[-] (I01) -- (I12);
    \draw[-] (I01) -- (I13);
    \draw[-] (I11) -- (I21);
    \draw[-] (I11) -- (I22);
    \draw[-] (I13) -- (I23);
    \draw[-] (I13) -- (I24);
    \draw[-] (I21) -- (I31);
    \draw[-] (I21) -- (I32);

    \draw[thick] (I31) circle (0.45cm);
\end{tikzpicture}\\
\begin{tikzpicture}[scale=0.7]
    \node (FAKE) at (0, 1) {\phantom{a}};
    \node (I01) at (0, 0) {$I_{0,1}$};
    \node (I11) at (-2, -1.3) {$I_{1,1}$};
    \node (I12) at (0, -1.3) {$I_{1,2}$};
    \node (I13) at (2, -1.3) {$I_{1,3}$};
    \node (I21) at (-3, -2.6) {$I_{2,1}$};
    \node (I22) at (-1, -2.6) {$I_{2,2}$};
    \node (I23) at (1, -2.6) {$I_{2,3}$};
    \node (I24) at (3, -2.6) {$I_{2,4}$};
    \node (I31) at (-4, -3.9) {$I_{3,1}$};
    \node (I32) at (-2, -3.9) {$I_{3,2}$};

    \draw[-] (I01) -- (I11);
    \draw[-] (I01) -- (I12);
    \draw[-] (I01) -- (I13);
    \draw[-] (I11) -- (I21);
    \draw[-] (I11) -- (I22);
    \draw[-] (I13) -- (I23);
    \draw[-] (I13) -- (I24);
    \draw[-] (I21) -- (I31);
    \draw[-] (I21) -- (I32);

    \draw[thick] (I32) circle (0.45cm);
    \draw[thick] (I31) ++(-0.5,-0.3) rectangle ++(1,0.7);
\end{tikzpicture}
    \caption{Tree diagrams illustrating the ordering of nodes for which projections are computed (circled) and the indices of previously computed contractions that can be exploited (boxed).}
    \label{fig: STTNN_indices_ordering}
\end{figure}
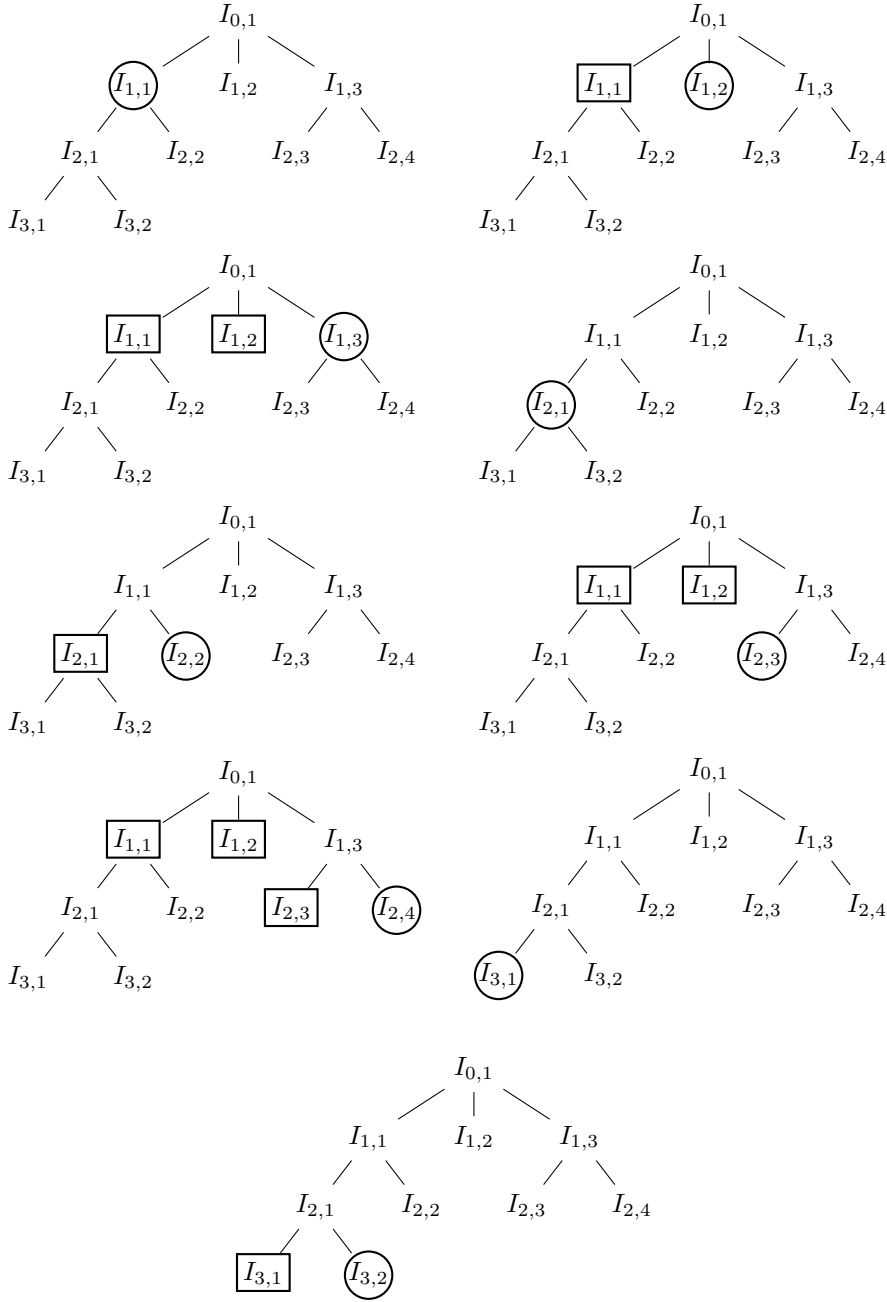
Overall, the STTNN algorithm is computationally more efficient than the TTNN algorithm when approximating dense tensors in the TTNN format. However, the situation becomes less straightforward when the tensor possesses a specific structure. For example, in the case of sparse tensors, a random contraction could yield a smaller yet dense tensor, potentially limiting the efficiency gains or even leading to a negative impact on performance. Another limitation of the STTNN algorithm is its reduced parallelizability. In TTNN, all sketchings can be computed independently, whereas in the sequential approach, the tensors must be processed one after the other. Nevertheless, this drawback does not compromise the streamability of the method, as the sketching phase remains linear in the input data.
\section{Error analysis}\label{sec:error analysis}

\noindent This section aims to show theoretical bounds on the accuracy of TTNN and STTNN approximants. In the first subsection, we provide a deterministic analysis, offering general guarantees that apply regardless of the specific choice of DRMs. In the second subsection, we delve into the case of Gaussian DRMs, where a more refined and precise analysis becomes feasible. Finally, in the last subsection, we discuss the error behavior of STTNN.
\subsection{Determistic analysis of TTNN}\label{sec:deterministic}
\par
Our approach mirrors the one used in \cite{kressner2023streaming} to derive bounds for STTA. In particular, we will utilize the following lemmas, closely adapted from the original source.

\begin{lemma}\label{lemma:removable_projections}
    Given $A\in \mathbb{R}^{m\times n}$, $B\in \mathbb{R}^{m \times q}$, $X\in \mathbb{R}^{n\times r}$ and $Y\in\mathbb{R}^{m\times (r+ p)}$, let $P = A X (Y^T {A} X)^\dagger Y^T $. Then, if $Y^T {A} X$ has full column rank
    \begin{equation}
        \|PB\|_F\leq \|{B}\|_F\left(1+ \|(Y^TQ)^\dagger (Y^T {Q}_\perp)\|_2\right),
    \end{equation}
    where $Q := \mathrm{orth}({AX})$.
\end{lemma}
\begin{proof}
Since $Y^TAX$ is full column rank, so is $AX$ and the QR decomposition $AX=QR$ produces an invertible factor $R$. Hence, we have
\begin{equation}
    P=AX(Y^TAX)^\dagger Y^T= QR(Y^TQR)^\dagger  Y^T = Q(Y^TQ)^\dagger Y^T.
\end{equation}
By completing $Q$ to a square orthogonal matrix $\begin{bmatrix}
    Q & Q_{\perp} 
\end{bmatrix}$ and using $QQ^T + Q_{\perp}Q_{\perp}^T=I$, we have
\begin{align*}
\|{P}{B}\|_F&=\|Q (Y^TQ )^\dagger Y^T {B} \|_F \\
&=\|(Y^TQ)^\dagger (Y^T Q_\perp)({Q}_\perp^T {B}) +(Y^TQ)^\dagger (Y^T {Q})(Q^T {B})\|_F \\
&\leq \|(Y^TQ)^\dagger (Y^T {Q}_\perp)({Q}_\perp^T {B})
\|_F+\|Q^T{B}\|_F \\
& \leq \|{B}\|_F\|(Y^TQ)^\dagger (Y^T {Q}_\perp)\|_2+\|{B}\|_F\\
& = \|{B}\|_F\left(1+\|(Y^TQ)^\dagger (Y^T {Q}_\perp)\|_2\right),
\end{align*}
where we used that $(Y^TQ)^\dagger(Y^TQ)=I$ by rank hypothesis. 
\end{proof} 
The following result is a generalization of \cite[Proposition~3.1]{kressner2023streaming}.

\begin{lemma}\label{lemma:sum_TTNN}
The approximation \(\widehat{\mathcal{T}}\) returned by TTNN satisfies the following inequality
\begin{equation}\label{eq:sum_TTNN}
\|\mathcal{T} - \widehat{\mathcal{T}}\|_F \leq \sum_{\ell=1}^L \sum_{k=1}^{K_\ell} \left\|\left(\mathop{\otimes}\limits_{j=1}^{K_L} P_{L,j}\right) \dots \left(\mathop{\otimes}\limits_{j=1}^{K_{\ell+1}} P_{\ell+1,j}\right) \left(
 I \otimes (I - P_{\ell,k}) \mathop{\otimes}\limits_{j=k+1}^{K_\ell} P_{\ell,j}  \right) \mathcal{T}^D\right\|_F,
\end{equation}
where the empty Kronecker product is understood to be omitted and the size of the identity matrices $I$ is such that all matrix products are well defined.
\end{lemma}

\begin{proof}
We begin by expressing the Frobenius norm of the difference between \(\mathcal{T}\) and \(\widehat{\mathcal{T}}\):

\[
\|\mathcal{T} - \widehat{\mathcal{T}}\|_F = \left\|\mathcal{T}^D - \left(\mathop{\otimes}\limits_{j=1}^{K_L} P_{L,j}\right) \dots \left(\mathop{\otimes}\limits_{j=1}^{K_1} P_{1,j}\right) \mathcal{T}^D \right\|_F.
\]

\noindent Expanding the term $P_{1,1}$, we get

\[
\|\mathcal{T} - \widehat{\mathcal{T}}\|_F = \left\|\mathcal{T}^D - \left(\mathop{\otimes}\limits_{j=1}^{K_L} P_{L,j}\right) \dots \left(\mathop{\otimes}\limits_{j=1}^{K_2} P_{2,j}\right) \left((I + P_{1,1}-\hspace{-1mm} I)\mathop{\otimes}\limits_{j=2}^{K_1}
P_{1,j}\right) \mathcal{T}^D \right\|_F.
\]

\noindent Using the triangle inequality, we can bound the expression by

\begin{align*}
\|\mathcal{T} - \widehat{\mathcal{T}}\|_F& \leq \left\|\mathcal{T}^D - \left(\mathop{\otimes}\limits_{j=1}^{K_L} P_{L,j}\right) \dots \left(\mathop{\otimes}\limits_{j=1}^{K_2} P_{2,j}\right) \left(I\otimes P_{1,2} \otimes \dots \otimes P_{1,K_1}\right) \mathcal{T}^D \right\|_F\\
&+ \left\|\left(\mathop{\otimes}\limits_{j=1}^{K_L} P_{L,j}\right) \dots \left(\mathop{\otimes}\limits_{j=1}^{K_2} P_{2,j}\right) \left((I-P_{1,1})\otimes P_{1,2} \otimes \dots \otimes P_{1,K_1} \right) \mathcal{T}^D \right\|_F.
\end{align*}

\noindent The second term on the right-hand side is one of the summands in the final inequality.
We apply the same reasoning to the first term. By similarly expanding \(P_{1,2} = I + P_{1,2} - I\) and using the subadditivity of the Frobenius norm, we obtain

\begin{align*}
\|\mathcal{T} - \widehat{\mathcal{T}}\|_F& \leq \left\|\mathcal{T}^D - \left(\mathop{\otimes}\limits_{j=1}^{K_L} P_{L,j}\right) \dots \left(\mathop{\otimes}\limits_{j=1}^{K_2} P_{2,j}\right) \left(
 I\otimes I \mathop{\otimes}\limits_{j=3}^{K_1} P_{1,j}\right) \mathcal{T}^D \right\|_F\\
& + \left\| \left(\mathop{\otimes}\limits_{j=1}^{K_L} P_{L,j}\right) \dots \left(\mathop{\otimes}\limits_{j=1}^{K_2} P_{2,j}\right) \left(I \otimes (I-P_{1,2})
 \mathop{\otimes}\limits_{j=3}^{K_1} P_{1,j}\right) \mathcal{T}^D \right\|_F\\
&+ \left\|\left(\mathop{\otimes}\limits_{j=1}^{K_L} P_{L,j}\right) \dots \left(\mathop{\otimes}\limits_{j=1}^{K_2} P_{2,j}\right) \left( (I-P_{1,1})
 \mathop{\otimes}\limits_{j=2}^{K_1} P_{1,j} \right) \mathcal{T}^D \right\|_F.
\end{align*}
This process is iterated  for all the projections at the first level of the index tree, until we reach the following bound
\begin{align*}
\|\mathcal{T} - \widehat{\mathcal{T}}\|_F & \leq \left\|\mathcal{T}^D - \left(\mathop{\otimes}\limits_{j=1}^{K_L} P_{L,j}\right) \dots \left(\mathop{\otimes}\limits_{j=1}^{K_2} P_{2,j}\right) \mathcal{T}^D \right\|_F\\ 
& +  \sum_{k=1}^{K_1} \left\|\left(\mathop{\otimes}\limits_{j=1}^{K_L} P_{L,j}\right) \dots \left(\mathop{\otimes}\limits_{j=1}^{K_2} P_{2,j}\right) \left( I \otimes (I-P_{1,k}) 
 \mathop{\otimes}\limits_{j=k+1}^{K_1} P_{1,j}\right) \mathcal{T}^D\right\|_F.
\end{align*}
Now, we can move to the second level, that is we can iterate the same reasoning to the first term of the previous bound, yielding to

\begin{align*}
\|\mathcal{T} - \widehat{\mathcal{T}}\|_F & \leq \left\|\mathcal{T}^D - \left(\mathop{\otimes}\limits_{j=1}^{K_L} P_{L,j}\right) \dots \left(\mathop{\otimes}\limits_{j=1}^{K_3} P_{3,j}\right) \mathcal{T}^D \right\|_F\\ 
& +  \sum_{\ell=1}^2\sum_{k=1}^{K_\ell} \left\|\left(\mathop{\otimes}\limits_{j=1}^{K_L} P_{L,j}\right) \dots \left(\mathop{\otimes}\limits_{j=1}^{K_{\ell + 1}} P_{\ell +1,j}\right) \left(
 I \otimes (I - P_{\ell,k}) \mathop{\otimes}\limits_{j=k+1}^{K_\ell} P_{\ell,j}\right) \mathcal{T}^D\right\|_F.
\end{align*}
We can now continue inductively until we have accounted for all levels. This leads us to the desired inequality.
\end{proof}

\noindent With the latter lemma, the analysis of the error becomes easier since the individual terms in \eqref{eq:sum_TTNN} are simpler to handle.
Therefore, we will now focus on one of these terms to conduct a more detailed examination.

\begin{lemma}\label{lemma:tau_TTNN} Using the same notation of Lemma~\ref{lemma:sum_TTNN}, let \( Q_{\ell, k} := \mathrm{orth}(\mathcal{T}^{\ell, k}X_{\ell, k}) \). Then, the following inequality holds
    \begin{align*}   &\left\|\left(\mathop{\otimes}\limits_{j=1}^{K_L} P_{L,j}\right) \dots \left(\mathop{\otimes}\limits_{j=1}^{K_{\ell+1}} P_{\ell+1,j}\right) 
    \left(I \otimes (I - P_{\ell,k})\mathop{\otimes}\limits_{j=k+1}^{K_\ell}P_{\ell,j}  \right) 
    \mathcal{T}^D\right\|_F \\
        &\qquad \leq \left(\prod_{j=1}^{K_L} \eta_{L,j}\right) \dots \left(\prod_{j=1}^{K_{\ell+1}} \eta_{\ell+1,j}\right) \left(\prod_{j=k+1}^{K_\ell} \eta_{\ell,j}\right) \|(I - P_{\ell,k})\mathcal{T}^{\ell,k}\|_F,
    \end{align*}
where
\[ \eta_{\ell,k} = 1 + \|(Y_{\ell, k}^T Q_{\ell, k})^\dagger Y_{\ell, k}^T Q_{\ell, k}^\perp\|_2 .\]
\end{lemma}

\begin{proof}
    Define \(\mathcal{B}\) as a tensor such that its vectorization is equal to the expression we wish to bound, i.e.,
    \[
    \mathcal{B}^D = \left(\mathop{\otimes}\limits_{j=1}^{K_L} P_{L,j}\right) \dots \left(\mathop{\otimes}\limits_{j=1}^{K_{\ell+1}} P_{\ell+1,j}\right) \left(I \otimes (I - P_{\ell,k})\mathop{\otimes}\limits_{j=k+1}^{K_\ell}P_{\ell,j} \right)  \mathcal{T}^D.
    \]
    Then, consider the matricization \(\mathcal{B}^{L,1}\), which can be expressed as \( P_{L,1} \mathcal{C}^{L,1} \) for some tensor \(\mathcal{C}\). By applying Lemma~\ref{lemma:removable_projections}, we have
    \[
    \|P_{L,1} \mathcal{C}^{L,1}\|_F \leq \eta_{L,1} \|\mathcal{C}\|_F.
    \]
    Similarly, the next matricization \(\mathcal{C}^{L,2}\) can be written as \( P_{L,2} \mathcal{D}^{L,2} \) for another tensor \(\mathcal{D}\), and again by Lemma~\ref{lemma:removable_projections}, we obtain
    \[
    \|P_{L,2} \mathcal{D}^{L,2}\|_F \leq \eta_{L,2} \|\mathcal{D}\|_F.
    \]
    Iterating this procedure for each subsequent projection yields the desired bound.
\end{proof}

We are ready to prove the main deterministic bound for TTNN.

\begin{thm}[Deterministic accuracy bound for TTNN]\label{thm:TTNN_deterministic_bound}
Let $\mathcal{T}\in\mathbb{R}^{n_1 \times\dots\times n_d} $
and $\widehat{\mathcal{T}}$ be the TTNN approximant of $\mathcal{T}$ with index tree $\mathcal{I}$, TTN ranks $\mathcal{R}$, TTN oversamples $\mathcal{P}$ and sketchings $X_{\ell, k}$, $Y_{\ell, k}$ defined in \eqref{sketch}. Then for any TTN ranks $\hat{\mathcal{R}}$ such that $\hat{\mathcal{R}}<\mathcal{R}$, setting
\\
 \begin{itemize}
     \item $\rho_{\ell, k} := \sqrt{1 + \|\widehat{V}_{\ell, k\perp}^{T}X_{\ell,k}(\widehat{V}_{\ell, k}^T X_{\ell, k})^\dagger\|_2^2}$,\\
     \item $\tau_{\ell, k}:= \sqrt{1 + \| (Y_{\ell, k}^T Q_{\ell ,k})^\dagger Y_{\ell, k}^TQ_{\ell, k\perp}\|_2^2}$,\\
     \item $\eta_{\ell,k} = 1 + \|(Y_{\ell, k}^T Q_{\ell, k})^\dagger Y_{\ell, k}^T Q_{\ell, k}^\perp\|_2$.
     \end{itemize}
     \vspace{2mm}
where $\widehat{V}_{\ell, k}$ is an orthogonal matrix with the first $\hat{r}_{\ell, k}<r_{\ell, k}$ right singular vectors of $\mathcal{T}^{\ell, k}$ and $Q_{\ell, k} = \mathrm{orth}(\mathcal{T}^{\ell, k} X_{\ell, k})$, the following holds
 \begin{equation}\label{eq:TTNN_deterministic_bound}
 \|\mathcal{T}-\hat{\mathcal{T}}\|_F \leq \|\mathcal{T}-\mathcal{T}_{\hat{\mathcal{R}}}\|_F\sum_{\ell=1}^{L}\sum_{k=1}^{K_\ell} \rho_{k,\ell}\tau_{k,\ell}\left(\prod_{j=1}^{K_L} \eta_{L,j}\right) \dots \left(\prod_{j=1}^{K_{\ell+1}} \eta_{\ell+1,j}\right) \left(\prod_{j=k+1}^{K_\ell} \eta_{\ell,j}\right).
 \end{equation}
 where $\mathcal{T}_{\hat{\mathcal{R}}}$ is any best TTN approximation of $\mathcal{T}$ of TTN rank $\hat{\mathcal{R}}$.
 \end{thm}
\begin{proof}
    By Lemma~\ref{lemma:sum_TTNN}
\[
\|\mathcal{T} - \widehat{\mathcal{T}}\|_F \leq \sum_{\ell=1}^L \sum_{k=1}^{K_\ell} \left\|\left(\mathop{\otimes}\limits_{j=1}^{K_L} P_{L,j}\right) \dots \left(\mathop{\otimes}\limits_{j=1}^{K_{\ell+1}} P_{\ell+1,j}\right) 
\left(I \otimes (I - P_{\ell,k}) \mathop{\otimes}\limits_{j=k+1}^{K_\ell} P_{\ell, j} \right) 
\mathcal{T}^D\right\|_F.
\]

Then, we can use Lemma~\ref{lemma:tau_TTNN} to bound each of the addends, obtaining
\[
\|\mathcal{T} - \widehat{\mathcal{T}}\|_F\leq \sum_{\ell=1}^L \sum_{k=1}^{K_\ell} \left(\prod_{j=1}^{K_L} \eta_{L,j}\right) \dots \left(\prod_{j=1}^{K_{\ell+1}} \eta_{\ell+1,j}\right) \left(\prod_{j=k+1}^{K_\ell} \eta_{\ell,j}\right) \|(I - P_{\ell,k})\mathcal{T}^{\ell,k}\|_F.
\]
The term $\|(I-P_{\ell, k})\mathcal{T}^{\ell, k}\|_F$, by Definition \ref{projectors}, satisfies 
\[
\|(I-P_{\ell, k})\mathcal{T}^{\ell, k}\|_F= \|\mathcal{T}^{\ell, k}- \mathcal{T}^{\ell, k}X_{\ell, k}(Y_{\ell, k}^T \mathcal{T}^{\ell, k}X_{\ell, k})^\dagger Y_{\ell, k}^T\|_F.
\]
The latter is the error of approximation of GN, which - according to \eqref{eq: HMT_error_bound} and \eqref{eq: GN_error_bound} - satisfies
    \[
    \|(I-P_{\ell, k})\mathcal{T}^{\ell, k}\|_F\leq \rho_{\ell, k} \tau_{\ell, k}\|\mathcal{T}^{\ell, k}-\mathcal{T}^{\ell, k}_{\hat{\mathcal{R}}}\|_F=\rho_{\ell, k} \tau_{\ell, k}\|\mathcal{T}-\mathcal{T}_{\hat{\mathcal{R}}}\|_F,
    \]
Putting it all together we showed that
\[
\|\mathcal{T} - \widehat{\mathcal{T}}\|_F\leq \sum_{\ell=1}^L \sum_{k=1}^{K_\ell} \|\mathcal{T}-\mathcal{T}_{\hat{\mathcal{R}}}\|_F\rho_{\ell, k}\tau_{\ell, k}\left(\prod_{j=1}^{K_L} \eta_{L,j}\right) \dots \left(\prod_{j=1}^{K_{\ell+1}} \eta_{\ell+1,j}\right) \left(\prod_{j=k+1}^{K_\ell} \eta_{\ell,j}\right) .
\]
\end{proof}
Note that if we define
\[
\rho := \max_{\ell, k} \rho_{\ell, k}, \quad \tau := \max_{\ell, k} \tau_{\ell, k}, \quad \eta := \max_{\ell, k} \eta_{\ell, k},
\]
then bound \eqref{eq:TTNN_deterministic_bound} simplifies to
\[
\|\mathcal{T} - \widehat{\mathcal{T}}\|_F \leq  \rho \tau \left(1 + \eta + \eta^2 + \dots + \eta^{|\mathcal{I}|-2}\right) \|\mathcal{T}-\mathcal{T}_{\hat{\mathcal{R}}}\|_F =  \rho \tau \left(\sum_{s=1}^{|\mathcal{I}|-2} \eta^s \right)       \|\mathcal{T}-\mathcal{T}_{\hat{\mathcal{R}}}\|_F.
\]

\noindent Theorem~\ref{thm:TTNN_deterministic_bound} suggests that the approximation constant of the TTNN method scales exponentially with the number of nodes $|\mathcal{I}|$.
However, as we will show in section~\ref{sec:experiments}, empirical evidence indicates that, in practice, the approximation constant scales more favorably as observed in \cite{bucci2024multilinear, kressner2023streaming}. 

\subsection{Probabilistic analysis of TTNN}\label{sec:probabilistic}
In this section, we provide an error analysis in the particular case of Gaussian DRMs $X_{\ell,k},Y_{\ell,k}$. For Gaussian DRMs, it is possible to extend the results of section \ref{sec:randomized matrix}. 
Our approach mirrors the one used in \cite[Theorem 3.2]{kressner2023streaming} to derive bounds for STTA. In particular, we will utilize the following lemma repeatedly in the proof of the expected error bound of Theorem \ref{thm: TTN expected value error}.

\begin{lemma}\label{lemma:removable_projections exp value}
    Given $A\in \mathbb{R}^{m\times n}$, $B\in \mathbb{R}^{m \times n}$, draw standard Gaussian matrices $X$ and $Y$ of size ${n\times r}$ and ${m\times (r+p)}$ respectively, and let $P = A X (Y^T {A} X)^\dagger Y^T $. Then, it holds
    \begin{equation}
        \mathbb{E}\|PB\|_F\leq \left(1+ \sqrt{\frac{r}{p-1}}\right)\|{B}\|_F.
    \end{equation}
\end{lemma}

\begin{proof}
As shown in Lemma \ref{lemma:removable_projections}, it holds
\begin{equation*}
    P=AX(Y^TAX)^\dagger Y^T= QR(Y^TQR)^\dagger  Y^T = Q(Y^TQ)^\dagger Y^T
\end{equation*}
and
\begin{equation*}
    \|{P}{B}\|_F \leq \|(Y^TQ)^\dagger (Y^T {Q}_\perp)({Q}_\perp^T {B})
\|_F+\|Q^T{B}\|_F.
\end{equation*}
By linearity of the expected value, we have
\begin{align}\label{eq: exp value projector}
\begin{split}
\mathbb{E}\|{P}{B}\|_F &\leq \mathbb{E}\|(Y^TQ)^\dagger (Y^T {Q}_\perp)({Q}_\perp^T {B})
\|_F+\mathbb{E}\|Q^T{B}\|_F \\
&\leq \mathbb{E}\|(Y^TQ)^\dagger (Y^T {Q}_\perp)({Q}_\perp^T {B})
\|_F+ \|{B}\|_F.
\end{split}
\end{align}

Since $Z_1:=Y^TQ$ and $Z_2:=Y^TQ_\perp$ are independent Gaussian matrices, we can apply Proposition 10.1 and Proposition 10.2 from \cite{halko2011finding} to have
\begin{align*}
    \mathbb{E}_{Z_1,Z_2}\| Z_1^\dagger Z_2 (Q_\perp^T B)\|_F
    &\leq \sqrt{\mathbb{E}_{Z_1,Z_2}\| Z_1^\dagger Z_2 (Q_\perp^T B)\|_F^2} \\
    & = \sqrt{\mathbb{E}_{Z_1} \left[ \mathbb{E}_{Z_2}\| Z_1^\dagger Z_2 (Q_\perp^T B)\|_F^2 \mid Z_1 \right] }\\
    &\leq  
    \sqrt{\mathbb{E}_{Z_1}\| Z_1^\dagger\|_F^2} \|Q_\perp^T B\|_F 
    \leq \sqrt{\frac{r}{p-1}}\|B\|_F.
\end{align*}
Insert this expression into \eqref{eq: exp value projector} to complete the proof.
\end{proof}

\begin{thm}\label{thm: TTN expected value error}
Let $\mathcal{T}\in\mathbb{R}^{n_1 \times\dots\times n_d} $
and $\widehat{\mathcal{T}}$ be the TTNN approximant of $\mathcal{T}$ with index tree $\mathcal{I}$, TTN ranks $\mathcal{R}$, TTN oversamples $\mathcal{P}$ and independent standard Gaussian DRMs $X_{\ell, k}$, $Y_{\ell, k}$ defined in \eqref{sketch}. Then for any TTN ranks $\hat{\mathcal{R}}$ such that $\hat{r}_{\ell,k}<r_{\ell,k}-1$, we have
\begin{align}
    \textstyle
     \mathbb{E} \|\mathcal{T} - \hat{\mathcal{T}}  \|_F 
    & \leq   \sum_{\ell=1}^{L} \sum_{k=1}^{K_\ell}  
   \left[\displaystyle\prod_{(t,s)>(\ell, k)} \hspace{-3mm}c_{t,s} \right] c_{\ell,k}' \sqrt{\sum_{i>\hat{r}_{\ell,k}} \sigma_i (\mathcal{T}^{\ell,k})^2} \\   
    & \leq \left(\sum_{\ell=1}^{L} \sum_{k=1}^{K_\ell}  
   \left[\displaystyle\prod_{(t,s)>(\ell, k)} \hspace{-3mm}c_{t,s} \right] c_{\ell,k}' \right) \| \mathcal{T}-\mathcal{T}_{\hat{\mathcal{R}}}\|_F,
\end{align}
where $\mathcal{T}_{\hat{\mathcal{R}}}$ is any best TTN approximation of $\mathcal{T}$ of TTN rank $\hat{\mathcal{R}}$ and $c_{\ell,k},c'_{\ell,k}$ are defined as
\begin{equation}
    c_{\ell,k}:= 1 + \sqrt{\frac{r_{\ell,k}}{p_{\ell,k}-1}},\qquad c'_{\ell,k}:=\sqrt{1 + \frac{r_{\ell,k}}{p_{\ell,k}-1}} \cdot \sqrt{1 + \frac{\hat{r}_{\ell,k}}{r_{\ell,k}-\hat{r}_{\ell,k}-1}}
\end{equation}
\end{thm}
\begin{proof}
By Lemma \ref{lemma:sum_TTNN}, we have
\begin{equation*}
    \|\mathcal{T} - \widehat{\mathcal{T}}\|_F \leq \sum_{\ell=1}^L \sum_{k=1}^{K_\ell} \left\|\left(\mathop{\otimes}\limits_{j=1}^{K_L} P_{L,j}\right) \dots \left(\mathop{\otimes}\limits_{j=1}^{K_{\ell+1}} P_{\ell+1,j}\right) \left(I \otimes
  (I - P_{\ell,k}) \mathop{\otimes}\limits_{j=k+1}^{K_\ell} P_{\ell,j}\right) \mathcal{T}^D\right\|_F,
\end{equation*}
and, by linearity of the expected value, we can study each one of these terms separately.
\par
For $\ell\in\{1,\dots,L\}$ and $k\in\{1,\dots,K_\ell\}$, it holds that
\begin{equation*}
    \left\|\left(\mathop{\otimes}\limits_{j=1}^{K_L} P_{L,j}\right) \dots \left(\mathop{\otimes}\limits_{j=1}^{K_{\ell+1}} P_{\ell+1,j}\right) \left(I \otimes
  (I - P_{\ell,k}) \mathop{\otimes}\limits_{j=k+1}^{K_\ell} P_{\ell,j}\right)\mathcal{T}^D\right\|_F= \| P_{L,1} M \|_F,
\end{equation*}
where $M$ is a proper reshaping of 
\begin{equation*}
    \left(\mathop{\otimes}\limits_{j=2}^{K_L} P_{L,j}\right) \dots \left(\mathop{\otimes}\limits_{j=1}^{K_{\ell+1}} P_{\ell+1,j}\right) \left(I \otimes
  (I - P_{\ell,k}) \mathop{\otimes}\limits_{j=k+1}^{K_\ell} P_{\ell,j}\right)\mathcal{T}^D.
\end{equation*}
and applying Lemma \ref{lemma:removable_projections exp value}
, we have
\begin{equation*}
    \mathbb{E}_{X_{L,1},Y_{L,1}}\| P_{L,1} M \|_F \leq c_{L,1} \| M \|_F.
\end{equation*}
Since all DRMs are independent, we can apply the law of total expectation to obtain
\begin{align*}  &\mathbb{E}_{X_{L,1},Y_{L,1},\dots,X_{\ell,k},Y_{\ell,k}}\left\|\left(\mathop{\otimes}\limits_{j=1}^{K_L} P_{L,j}\right) \dots \left(\mathop{\otimes}\limits_{j=1}^{K_{\ell+1}} P_{\ell+1,j}\right) \left(I \otimes
  (I - P_{\ell,k}) \mathop{\otimes}\limits_{j=k+1}^{K_\ell} P_{\ell,j}\right)\mathcal{T}^D\right\|_F  \\
&\qquad \leq\mathbb{E}_{X_{L,2},Y_{L,2},\dots,X_{\ell,k},Y_{\ell,k}}\left[\mathbb{E}_{X_{L,1},Y_{L,1}}\left\|P_{L,1}M\right\|_F| X_{L,2},Y_{L,2},\dots,X_{\ell,k},Y_{\ell,k}\right] \\
&\qquad \leq c_{L,1} \mathbb{E}_{X_{L,2},Y_{L,2},\dots,X_{\ell,k},Y_{\ell,k}}\left\|M\right\|_F. \\
\end{align*}
We can now repeat the argument, by reshaping at each iteration, to reach
\begin{align}\label{eq: exp value TTN last step}
\begin{split}   \mathbb{E}&\left\|\left(\mathop{\otimes}\limits_{j=1}^{K_L} P_{L,j}\right) \dots \left(\mathop{\otimes}\limits_{j=1}^{K_{\ell+1}} P_{\ell+1,j}\right) \left(I \otimes
  (I - P_{\ell,k}) \mathop{\otimes}\limits_{j=k+1}^{K_\ell} P_{\ell,j}\right) \mathcal{T}^D\right\|_F \\
&\leq [\displaystyle\prod_{\substack{\ell+1\leq t \leq L \\
    1\leq s \leq K_t }} c_{t,s} ][\prod_{j=k+1}^{K_\ell} c_{\ell,j} ] \mathbb{E}_{X_{\ell,k},Y_{\ell,k}}\|(I-P_{\ell,k})\mathcal{T}^{\ell,k}\|_F\\
    & = 
   \left[\displaystyle\prod_{(t,s)>(\ell, k)} \hspace{-3mm}c_{t,s} \right] \mathbb{E}_{X_{\ell,k},Y_{\ell,k}}\|(I-P_{\ell,k})\mathcal{T}^{\ell,k}\|_F.
\end{split} 
\end{align}
The last expected value to compute is the one corresponding to GN error approximation for which it holds
\begin{equation*}
    \mathbb{E}_{X_{\ell,k},Y_{\ell,k}}\|(I-P_{\ell,k})\mathcal{T}^{\ell,k}\|_F \leq c'_{\ell,k} \sqrt{\sum_{i>\hat{r}_{\ell,k}} \sigma_i (\mathcal{T}^{\ell,k})^2} \leq c'_{\ell,k}  \| \mathcal{T} -\mathcal{T}_{\hat{\mathcal{R}}}\|_F.
\end{equation*}
The first inequality holds by \eqref{eq: GN_error_bound_expectation} and the second one by \cite[Theorem 11.6]{book_hackbush}.
Insert this inequality into \eqref{eq: exp value TTN last step} and combine the resulting inequality with Lemma \ref{lemma:sum_TTNN} to conclude.
\end{proof}
\subsection{Analysis of STTNN}\label{sec:sequential deterministic}
By looking at \eqref{TTNN approximant projector} and \eqref{STTNN approximant projector}, we observe that the structure of the approximations provided by TTNN and STTNN is remarkably similar. In particular, the deterministic analysis of TTNN, presented in Theorem \ref{thm:TTNN_deterministic_bound}, can be adapted to STTNN with minimal modifications. The only adjustment lies in accounting for the slightly different projections used, as they involve the contracted tensors $\mathcal{T}^{\ell, k}_{S_{\ell, k}}$, see Equation \eqref{STTNN approximant projector}. In particular, the following holds
\begin{thm}[Deterministic accuracy bound for STTNN]\label{thm:STTNN_deterministic_bound}
Let $\mathcal{T}\in\mathbb{R}^{n_1 \times\dots\times n_d} $
and $\widehat{\mathcal{T}}$ be the STTNN approximant of $\mathcal{T}$ with index tree $\mathcal{I}$, TTN ranks $\mathcal{R}$, TTN oversamples $\mathcal{P}$, sketchings $Y_{\ell, k}$ as defined in \eqref{sketch}, and sketchings $X_{\ell, k}$ with an appropriate number of rows and $r_{\ell, k}$ columns to ensure that \eqref{eq: general_seq_projection} is well-defined. Then for any TTN ranks $\hat{\mathcal{R}}$ such that $\hat{\mathcal{R}}<\mathcal{R}$, setting
\\
 \begin{itemize}
     \item $\rho_{\ell, k} := \sqrt{1 + \|\widehat{V}_{\ell, k\perp}^{T}X_{\ell,k}(\widehat{V}_{\ell, k}^T X_{\ell, k})^\dagger\|_2^2}$,\\
     \item $\tau_{\ell, k}:= \sqrt{1 + \| (Y_{\ell, k}^T Q_{\ell ,k})^\dagger Y_{\ell, k}^TQ_{\ell, k\perp}\|_2^2}$,\\
     \item $\eta_{\ell,k} = 1 + \|(Y_{\ell, k}^T Q_{\ell, k})^\dagger Y_{\ell, k}^T Q_{\ell, k}^\perp\|_2$.
     \end{itemize}
     \vspace{2mm}
where $\widehat{V}_{\ell, k}$ is an orthogonal matrix with the first $\hat{r}_{\ell, k}<r_{\ell, k}$ right singular vectors of $\mathcal{T}^{\ell, k}$ and $Q_{\ell, k} = \mathrm{orth}(\mathcal{T}^{\ell, k}_{S_{\ell,k}} X_{\ell, k})$, the following holds
 \begin{equation}\label{eq:STTNN_deterministic_bound}
 \|\mathcal{T}-\hat{\mathcal{T}}\|_F \leq \|\mathcal{T}-\mathcal{T}_{\hat{\mathcal{R}}}\|\sum_{\ell=1}^{L}\sum_{k=1}^{K_\ell} \rho_{k,\ell}\tau_{k,\ell}\left(\prod_{j=1}^{K_L} \eta_{L,j}\right) \dots \left(\prod_{j=1}^{K_{\ell+1}} \eta_{\ell+1,j}\right) \left(\prod_{j=k+1}^{K_\ell} \eta_{\ell,j}\right).
 \end{equation}
 where $\mathcal{T}_{\hat{\mathcal{R}}}$ is any best TTN approximation of $\mathcal{T}$ of TTN rank $\hat{\mathcal{R}}$.
 \end{thm}
Obtaining an expected error bound, even with Gaussian DRMs, is a more complicated task due to the nature of projectors. In particular, the projectors involve sketches as the one in \eqref{eq: general_seq_projection}, in which the term $\mathcal{T}_{S_{\ell, k}}^{\ell, k}X_{\ell, k}$ appears. This term can be explicitly written as
\begin{equation}
    \mathcal{T}^{\ell, k}(I \otimes Y_{\ell_s, k_s}\otimes \dots \otimes Y_{\ell_1,k_1 } \otimes I)X_{\ell, k}.
\end{equation}
where we assumed, without loss of generality, that the partition $S_{\ell, k}$ is made of consecutive indices. This structure does not allow us to apply the same approach of Theorem \ref{thm: TTN expected value error} to obtain an expected error bound.

\section{Structured sketchings for TTNN and STTNN}\label{sec:khatri-rao}
So far we presented TTNN and STTNN as algorithms for approximate nonstructured tensors in the TTN format. However, in the vast majority of applications, these tensors are already given in TTN format. The goal in such cases is to obtain a representation with lower TTN ranks. A practical example is the low-rank compression of a sum of low-rank TTN tensors.

For these applications, it is crucial to use sketching techniques that leverage the TTN structure of the tensor to reduce the computational cost, which would otherwise be prohibitive. For the TT format, for example, \cite{rounding_TT, kressner2023streaming} suggest using TT-Gaussian matrices that are DRMs in TT format with each core containing i.i.d. Gaussian entries. A natural extension of this idea would be to construct DRMs in the TTN format, made of transfer tensors with i.i.d. Gaussian entries, in such a way that the operations involved can be hierarchically split across the cores.

Similar approaches to structured sketching in the TTN format have been explored in \cite{ma2022cost, ahle2020oblivious, mahankali2022near}. Here, we focus on Khatri-Rao embeddings, which demonstrate excellent practical performance. However, in worst-case scenarios, they may require a sketch size that scales exponentially with the order of the tensor to achieve satisfactory results \cite{ahle2020oblivious}.

Given a tensor $\mathcal{T}\in \mathbb{R}^{n_1\times \dots \times n_d}$ in TTN format with index tree $\mathcal{I}$, to construct the structured sketchings we define the matrices
\phantom{a}\\
\begin{itemize}
\setlength\itemsep{1em}
    \item $X_{i}\in \mathbb{R}^{n_i\times r}$, for $i=1, \dots, d$,
    \item $Y_{i}\in \mathbb{R}^{n_i \times (r+p)}$, for $i=1, \dots, d$,
\end{itemize}
\phantom{a}\\
and then we construct the sketchings $X_{\ell, k}$ and $Y_{\ell, k}$ of the TTNN approximant as
\phantom{a}\\
\begin{itemize}
\setlength\itemsep{1em}
    \item $X_{\ell, k}= \mathop{\odot}\limits_{i\in D\backslash I_{\ell, k}}\hspace{-3mm} X_i \in \mathbb{R}^{n_{D\backslash I_{\ell, k}}\times r}$,
    \item $Y_{\ell, k}= \mathop{\odot}\limits_{i\in I_{\ell, k}}\hspace{-1mm} Y_i \in \mathbb{R}^{n_{I_{\ell, k}}\times (r+p)}$,
\end{itemize}
\phantom{a}\\
where with $\odot$ we denote the column-wise Khatri-Rao product.

Below, in Figure \ref{fig: khatri-rao sketchings}, we provide a graphical illustration of how to compute efficiently the sketchings $\Omega_{1,1}$ and $\Psi_{1,1}$, defined in section \ref{sec:TTNN}, for a 6D tensor $\mathcal{T}$ in TTN format. The tensor is structured according to the index tree depicted in Figure \ref{fig: example of index tree}. Notably, in this specific case, the involved matrices take the following form
\phantom{a}\\
\begin{itemize}
\setlength\itemsep{1em}
    \item $\Omega_{1,1} = Y_{1,1}^T\mathcal{T}^{1,1}X_{1,1}=\left(Y_1\odot Y_2 \odot Y_3\right)^T\mathcal{T}^{1,1}\left(X_4\odot X_5 \odot X_6\right)$,
    \item $\Psi_{1,1} = (Y_{2,1}^T\otimes Y_{2,2}^T)\mathcal{T}^{1,1}X_{1,1} = \left((Y_1\odot Y_2)\otimes Y_3\right)^T\mathcal{T}^{1,1}\left(X_4\odot X_5 \odot X_6\right)$
\end{itemize}
\phantom{a}\\
and entrywise they can be written as
\phantom{a}\\
\begin{itemize}
\setlength\itemsep{1em}
    \item $[\Omega_{1,1}]_{i,j} =\left(y_1^{(i)}\otimes y_2^{(i)} \otimes y_3^{(i)}\right)^T\mathcal{T}^{1,1}\left(x_4^{(j)}\otimes x_5^{(j)} \otimes x_6^{(j)}\right)$,
    \item $[\Psi_{1,1}]_{(i_1,i_2),j} =\left(y_1^{(i_1)}\otimes y_2^{(i_1)} \otimes y_3^{(i_2)}\right)^T\mathcal{T}^{1,1}\left(x_4^{(j)}\otimes x_5^{(j)} \otimes x_6^{(j)}\right)$,
\end{itemize}
\phantom{a}\\
where we used $x_s^{(k)}$ and $y_s^{(k)}$ to denote the $k$th columns of $X_s$ and $Y_s$ respectively and, with a slight abuse of notation, we have partitioned the indices of $\Psi_{1,1}$ to clarify the components involved.
\phantom{a}\\
\begin{figure}
    \centering
\begin{tikzpicture}[scale=0.7]
    
    \node (FAKE) at (0, 1) {\phantom{a}};
    \node (I01) at (0, 0) {$I_{0,1}$};
    \node (I11) at (-2, -1.3) {$I_{1,1}$};
    \node (I12) at (0, -1.3) {$I_{1,2}$};
    \node (I13) at (2, -1.3) {$I_{1,3}$};
    \node (I21) at (-3, -2.6) {$I_{2,1}$};
    \node (I22) at (-1, -2.6) {$I_{2,2}$};
    \node (I23) at (1, -2.6) {$I_{2,3}$};
    \node (I24) at (3, -2.6) {$I_{2,4}$};
    \node (I31) at (-4, -3.9) {$I_{3,1}$};
    \node (I32) at (-2, -3.9) {$I_{3,2}$};

    \node (X1) at (-4, -5.5) {$y_1^{(i)}$};
    \node (X2) at (-2, -5.5) {$y_2^{(i)}$};
    \node (Y3) at (-1, -4.3) {$y_3^{(i)}$};
    \node (X4) at (0, -3) {$x_4^{(j)}$};
    \node (X5) at (1, -4.3) {$x_5^{(j)}$};
    \node (X6) at (3, -4.3) {$x_6^{(j)}$};

    \draw[-] (I01) -- (I11);
    \draw[-] (I01) -- (I12);
    \draw[-] (I01) -- (I13);
    \draw[-] (I11) -- (I21);
    \draw[-] (I11) -- (I22);
    \draw[-] (I13) -- (I23);
    \draw[-] (I13) -- (I24);
    \draw[-] (I21) -- (I31);
    \draw[-] (I21) -- (I32);

    \draw[thick, fill=cyan!50] (I01) circle (0.45cm);
    \draw[thick, fill=cyan!50] (I11) circle (0.45cm);
    \draw[thick, fill=cyan!50] (I12) circle (0.45cm);
    \draw[thick, fill=cyan!50] (I13) circle (0.45cm);
    \draw[thick, fill=cyan!50] (I11) circle (0.45cm);
    \draw[thick, fill=cyan!50] (I22) circle (0.45cm);
    \draw[thick, fill=cyan!50] (I21) circle (0.45cm);
    \draw[thick, fill=cyan!50] (I23) circle (0.45cm);
    \draw[thick, fill=cyan!50] (I24) circle (0.45cm);
    \draw[thick, fill=cyan!50] (I31) circle (0.45cm);
    \draw[thick, fill=cyan!50] (I32) circle (0.45cm);
    
    \draw[thick] (X1) circle (0.5cm);
    \draw[thick] (X2) circle (0.5cm);
    \draw[thick] (Y3) circle (0.5cm);
    \draw[thick] (X4) circle (0.5cm);
    \draw[thick] (X5) circle (0.5cm);
    \draw[thick] (X6) circle (0.5cm);

    \draw[-] (I31) -- (X1);
    \draw[-] (I32) -- (X2);
    \draw[-] (I22) -- (Y3);
    \draw[-] (I12) -- (X4);
    \draw[-] (I23) -- (X5);
    \draw[-] (I24) -- (X6);
\end{tikzpicture}~\begin{tikzpicture}
    \node (DUMMY) at (0,0) {\phantom{a}};
    \node (DUMMY2) at (0,-5) {\phantom{a}};
\end{tikzpicture}~\begin{tikzpicture}[scale=0.7]
    \node (FAKE) at (0, 1) {\phantom{a}};
    \node (I01) at (0, 0) {$I_{0,1}$};
    \node (I11) at (-2, -1.3) {$I_{1,1}$};
    \node (I12) at (0, -1.3) {$I_{1,2}$};
    \node (I13) at (2, -1.3) {$I_{1,3}$};
    \node (I21) at (-3, -2.6) {$I_{2,1}$};
    \node (I22) at (-1, -2.6) {$I_{2,2}$};
    \node (I23) at (1, -2.6) {$I_{2,3}$};
    \node (I24) at (3, -2.6) {$I_{2,4}$};
    \node (I31) at (-4, -3.9) {$I_{3,1}$};
    \node (I32) at (-2, -3.9) {$I_{3,2}$};

    \node (X1) at (-4, -5.5) {$y_1^{(i_1)}$};
    \node (X2) at (-2, -5.5) {$y_2^{(i_1)}$};
    \node (Y3) at (-1, -4.3) {$y_3^{(i_2)}$};
    \node (X4) at (0, -3) {$x_4^{(j)}$};
    \node (X5) at (1, -4.3) {$x_5^{(j)}$};
    \node (X6) at (3, -4.3) {$x_6^{(j)}$};

    \draw[-] (I01) -- (I11);
    \draw[-] (I01) -- (I12);
    \draw[-] (I01) -- (I13);
    \draw[-] (I11) -- (I21);
    \draw[-] (I11) -- (I22);
    \draw[-] (I13) -- (I23);
    \draw[-] (I13) -- (I24);
    \draw[-] (I21) -- (I31);
    \draw[-] (I21) -- (I32);

    \draw[thick, fill=cyan!50] (I01) circle (0.45cm);
    \draw[thick, fill=cyan!50] (I11) circle (0.45cm);
    \draw[thick, fill=cyan!50] (I12) circle (0.45cm);
    \draw[thick, fill=cyan!50] (I13) circle (0.45cm);
    \draw[thick, fill=cyan!50] (I11) circle (0.45cm);
    \draw[thick, fill=cyan!50] (I22) circle (0.45cm);
    \draw[thick, fill=cyan!50] (I21) circle (0.45cm);
    \draw[thick, fill=cyan!50] (I23) circle (0.45cm);
    \draw[thick, fill=cyan!50] (I24) circle (0.45cm);
    \draw[thick, fill=cyan!50] (I31) circle (0.45cm);
    \draw[thick, fill=cyan!50] (I32) circle (0.45cm);
    
    \draw[thick] (X1) circle (0.55cm);
    \draw[thick] (X2) circle (0.55cm);
    \draw[thick] (Y3) circle (0.55cm);
    \draw[thick] (X4) circle (0.5cm);
    \draw[thick] (X5) circle (0.5cm);
    \draw[thick] (X6) circle (0.5cm);

    \draw[-] (I31) -- (X1);
    \draw[-] (I32) -- (X2);
    \draw[-] (I22) -- (Y3);
    \draw[-] (I12) -- (X4);
    \draw[-] (I23) -- (X5);
    \draw[-] (I24) -- (X6);
\end{tikzpicture}
    \caption{Graphical illustration of how to use Khatri-Rao embeddings to compute the entries of $\Omega_{1,1}$ (left) and $\Psi_{1,1}$ (right) exploiting the tree structure. }
    \label{fig: khatri-rao sketchings}
\end{figure}
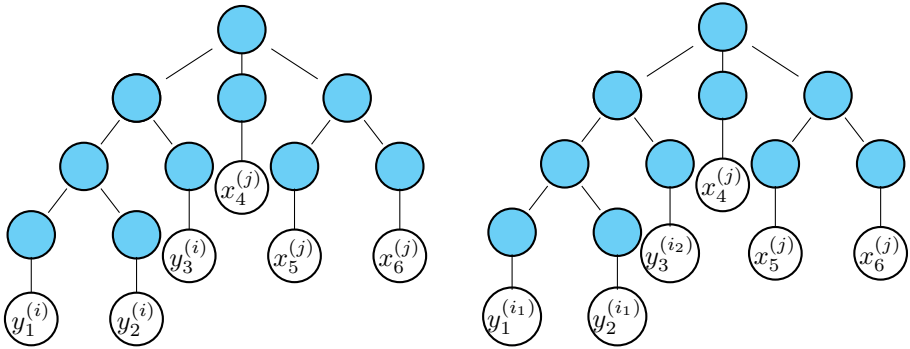
Note that, with these DRMs, the cost of computing a transfer tensor is linear in the $n_i$s. Moreover, many of the contractions can be recycled to calculate the other transfer tensors.

\section{Numerical experiments}\label{sec:experiments}
In this section, we evaluate the performance and accuracy of the TTNN and STTNN algorithms. First, we demonstrate their effectiveness by computing TTN approximations of dense tensors. Subsequently, we shift our focus to the recompression and rounding of tensors already represented in the TTN format. For these analyses, we will consider the index tree in Figure \ref{fig: example of index tree}. Since the tree represents a 6-dimensional object, we cannot address tensors with large mode sizes in the dense case. Thus, we focus on larger tensors only in the second part of the analysis. All numerical experiments were performed in MATLAB version 2023b on a laptop
with 16GB of system memory and the code used for the numerical experiments is available at \href{https://github.com/alb95/TTNN}{https://github.com/alb95/TTNN}.
As a first example, we analyze the accuracy of the TTNN and the STTNN algorithms on a 6D Hilbert tensor $\mathcal{H}$, i.e. $\mathcal{H}(i_1, \dots, i_6) = \frac{1}{1+i_1 + \dots + i_6}$ with mode sizes $20$. In particular, we compute the relative error of approximation in the Frobenius norm by varying the TTN ranks $\mathcal{R}$ and fixing the oversampling parameters $\mathcal{P}$ of the approximant.
For simplicity, the TTN ranks involved in the approximations are set to the same value $r$ and the oversampling parameters to $p=3$. 
Additionally, to provide a benchmark for comparison, we include the approximation error provided by adapting the hierarchically SVD \cite{grasedyck2010hierarchical} to the TTN setting (TTN-SVD). We also include a randomized variant of the TTN-SVD, referred to as TTN-HMT, where the standard SVDs are replaced by the HMT algorithm. A similar approach for the TT format is presented in \cite{kressner2023streaming, huber2017randomized}. Since some methods involve randomness, we perform 30 trials and show the mean relative error $\| \mathcal{T}-\hat{\mathcal{T}}\|_F/ \| \mathcal{T}\|_F$ as well as average running time. For consistency, we use the same right sketch matrices $X_{\ell, k}$ in both TTNN and TTN-HMT methods. The results of this experiment are reported in Figure~\ref{fig:hilbert_toy_example}.
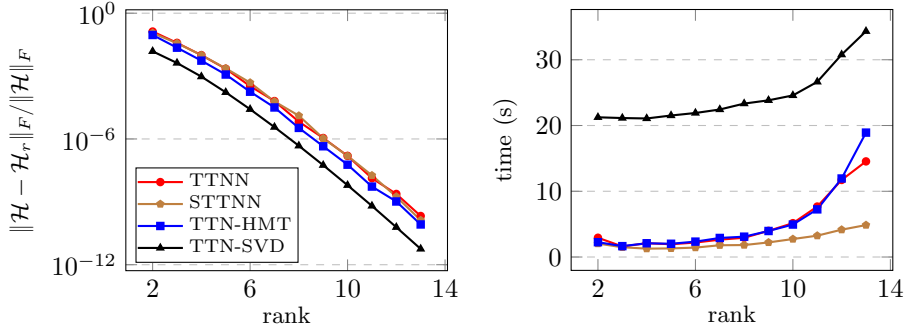
\begin{figure}
    \centering
    \begin{tikzpicture}
\begin{semilogyaxis}[
    x label style={at={(axis description cs:0.5,-0.1)},anchor=north},
    xlabel={\small{rank}},
    ylabel={\small{$\|\mathcal{H}-\mathcal{H}_r\|_F/\|\mathcal{H}\|_F$}},
    legend pos=north east,
    xtick = {2,6, 10, 14},
    width=.45\linewidth,
    ymajorgrids=true,
    grid style=dashed,
    legend pos = south west,
    mark size = 1.2pt,
    legend style={font=\scriptsize, cells={anchor=west}, inner sep=0.3pt,row sep=-2pt}, 
    restrict x to domain=0:13
]
    \addplot [color = red, style = thick,mark =*,  mark size=1.2pt] table [col sep=comma, x index = 0, y index = 1] {TTNN_data/6D_Hilbert_avaraged_accuracy.csv};
    \addplot [color = brown, style = thick,mark =pentagon*,  mark size=1.2pt] table [col sep=comma, x index = 0, y index = 2] {TTNN_data/6D_Hilbert_avaraged_accuracy.csv};
    \addplot [color = blue, style = thick,mark =square*,  mark size=1.2pt] table [col sep=comma, x index = 0, y index = 4] {TTNN_data/6D_Hilbert_avaraged_accuracy.csv};
    \addplot  [color = black, style = thick,mark =triangle*,  mark size=1.2pt] table [col sep=comma, x index = 0, y index = 3] {TTNN_data/6D_Hilbert_avaraged_accuracy.csv};
    \legend{\scriptsize{TTNN},\scriptsize{STTNN}, \scriptsize{TTN-HMT}, \scriptsize{TTN-SVD}}
\end{semilogyaxis}
\end{tikzpicture}~   \begin{tikzpicture}
\begin{axis}[
    x label style={at={(axis description cs:0.5,-0.1)},anchor=north},
    xlabel={\small{rank}},
    ylabel={\small{time (s)}},
    legend pos=north east,
    xtick = {2,6, 10, 14},
    width=.45\linewidth,
    ymajorgrids=true,
    grid style=dashed,
    legend pos = south west,
    mark size = 1.2pt,
    legend style={font=\scriptsize, cells={anchor=west}, inner sep=0.3pt,row sep=-2pt}, 
    restrict x to domain=0:13
]
    \addplot [color = red, style = thick,mark =*,  mark size=1.2pt] table [col sep=comma, x index = 0, y index = 1] {TTNN_data/6D_Hilbert_avaraged_times.csv};
    \addplot [color = brown, style = thick,mark =pentagon*,  mark size=1.2pt] table [col sep=comma, x index = 0, y index = 2] {TTNN_data/6D_Hilbert_avaraged_times.csv};
    \addplot [color = blue, style = thick,mark =square*,  mark size=1.2pt] table [col sep=comma, x index = 0, y index = 4] {TTNN_data/6D_Hilbert_avaraged_times.csv};
    \addplot  [color = black, style = thick,mark =triangle*,  mark size=1.2pt] table [col sep=comma, x index = 0, y index = 3] {TTNN_data/6D_Hilbert_avaraged_times.csv};
\end{axis}
\end{tikzpicture}
    \caption{Frobenius error of approximation (left) and running time (right) obtained by the TTNN algorithm with the index tree in Figure \ref{fig: example of index tree} and Gaussian sketchings for different values of $r$ on the 6D Hilbert tensor of size $20\times 20 \times 20 \times 20 \times 20 \times 20$.}
    \label{fig:hilbert_toy_example}
\end{figure}
The experiment demonstrates that all the randomized algorithms perform very well. Both TTNN and STTNN achieve comparable levels of accuracy, while TTN-HMT proves to be slightly more accurate. This outcome aligns with our expectations and is consistent with observations in the matrix case. However, the sequential variant outperforms the others by a significant margin in terms of runtime, establishing itself as a practical and efficient method for tree tensor network compression in both streaming and non-streaming settings.

In the second experiment, we evaluate the performance of the TTNN method for TTN rounding of a tensor and compare it once again with TTN-SVD and TTN-HMT, both appropriately adapted to the TTN format. We exclude STTNN from this comparison, as its computational advantage primarily stems from iterating on smaller tensors. However, since these tensors are already in a compressed format, this approach does not offer significant benefits in this context.
In the second experiment, we evaluate the performance of the TTNN method for TTN rounding of a tensor and compare it with TTN-SVD and TTN-HMT, both appropriately adapted to the TTN format. We exclude STTNN from this comparison, as its computational advantage primarily arises from iterating on smaller tensors. However, since these tensors are already in a compressed format, this approach does not yield significant benefits in this scenario.

The experiment involves recompressing synthetic tensors provided in TTN format. The internal core tensors are generated using an orthogonal CP decomposition. Specifically, each core tensor is constructed as a superdiagonal tensor (i.e., a tensor with non-zero entries only along the diagonal) with entries $\sigma_i$ that follow a prescribed decay. These core tensors are then multiplied by a set of Haar-distributed orthogonal matrices along each mode. This process ensures that the TTN cores retain a structured form, where the $\sigma_i$ values control the magnitude of the components. The leaf matrices are instead Haar-distributed orthogonal matrices.
The mode sizes of these tensors are set to $500$, and the TTN ranks are fixed at $70$. We analyze three different decay patterns for the $\sigma_i$ values: quadratic ($\sigma_i = 1/i^2$), cubic ($\sigma_i = 1/i^3$), and exponential ($\sigma_i = 1/2^i$). We compute the relative error of approximation in the Frobenius norm and the running time by varying the TTN ranks $\mathcal{R}$ and fixing the oversampling parameters $\mathcal{P}$. 
The TTN ranks involved in the approximations are set to the same value $r$ and the oversampling parameters to $p=10$. 
In the experiments, presented in Figure \ref{fig: TTN rounding}, we report the averaged quantities over 30 attempts. 
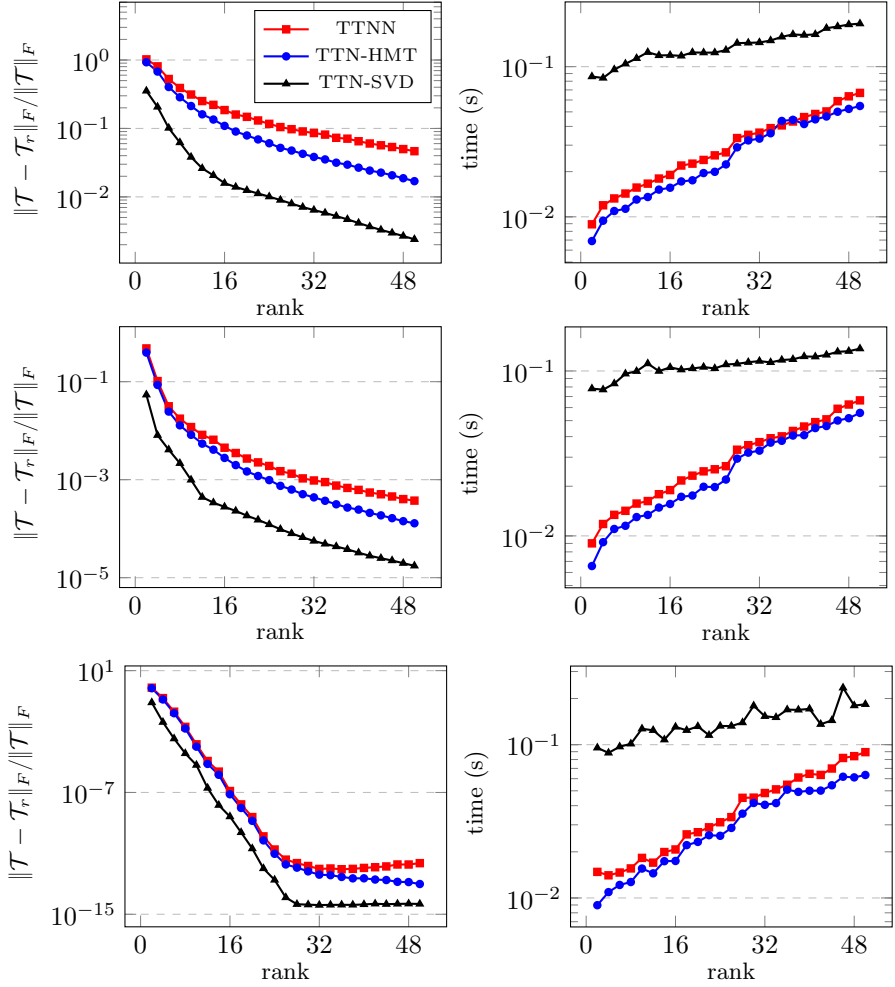
\begin{figure}
\centering
\begin{tikzpicture}
\begin{semilogyaxis}[
    x label style={at={(axis description cs:0.5,-0.1)},anchor=north},
    xlabel={\small{rank}},
    ylabel={\small{ $\|\mathcal{T}-\mathcal{T}_r\|_F/\|\mathcal{T}\|_F$}},
    ymajorgrids=true,
    grid style=dashed,
    ymax = 7,
    xtick = {0,16,32,48},
    width=.45\linewidth,
    legend pos = north east,
    mark size = 1.2pt,
    restrict x to domain=0:125
]
\hspace*{0.5cm}
\addplot [color = red, style = thick,mark = square*,  mark size=1.2pt] table [col sep=comma, x index = 0, y index = 1,x filter/.code={\pgfmathparse{mod(\coordindex,50)==0 ? x : inf}}] {TTNN_data/toy_example_avarage_accuracy_quadratic.csv};
\addplot [color = blue, style = thick,mark = *,  mark size=1.2pt] table [col sep=comma, x index = 0, y index = 2,x filter/.code={\pgfmathparse{mod(\coordindex,50)==0 ? x : inf}}] {TTNN_data/toy_example_avarage_accuracy_quadratic.csv};
\addplot [color = black, style = thick,mark =triangle*,  mark size=1.2pt] table [col sep=comma, x index = 0, y index = 3,x filter/.code={\pgfmathparse{mod(\coordindex,50)==0 ? x : inf}}] {TTNN_data/toy_example_avarage_accuracy_quadratic.csv};
\legend{\scriptsize{TTNN}, \scriptsize{TTN-HMT}, \scriptsize{TTN-SVD}}
\end{semilogyaxis}
\end{tikzpicture}~\begin{tikzpicture}
\begin{semilogyaxis}[
    x label style={at={(axis description cs:0.5,-0.1)},anchor=north},
    xlabel={\small{rank}},
    ylabel={\small{ time (s)}},
    ymajorgrids=true,
    grid style=dashed,
    xtick = {0,16,32,48},
    width=.45\linewidth,
    mark size = 1.2pt,
    restrict x to domain=0:125
]
\hspace*{0.5cm}
\addplot [color = red, style = thick,mark = square*,  mark size=1.2pt] table [col sep=comma, x index = 0, y index = 1,x filter/.code={\pgfmathparse{mod(\coordindex,50)==0 ? x : inf}}] {TTNN_data/toy_example_avarage_times_quadratic.csv};
\addplot [color = blue, style = thick,mark = *,  mark size=1.2pt] table [col sep=comma, x index = 0, y index = 2,x filter/.code={\pgfmathparse{mod(\coordindex,50)==0 ? x : inf}}] {TTNN_data/toy_example_avarage_times_quadratic.csv};
\addplot [color = black, style = thick,mark =triangle*,  mark size=1.2pt] table [col sep=comma, x index = 0, y index = 3,x filter/.code={\pgfmathparse{mod(\coordindex,50)==0 ? x : inf}}] {TTNN_data/toy_example_avarage_times_quadratic.csv};
\end{semilogyaxis}
\end{tikzpicture}\\
\begin{tikzpicture}
\begin{semilogyaxis}[
    x label style={at={(axis description cs:0.5,-0.1)},anchor=north},
    xlabel={\small{rank}},
    ylabel={\small{ $\|\mathcal{T}-\mathcal{T}_r\|_F/\|\mathcal{T}\|_F$}},
    ymajorgrids=true,
    grid style=dashed,
    xtick = {0,16,32,48},
    width=.45\linewidth,
    mark size = 1.2pt,
    restrict x to domain=0:125
]
\hspace*{0.5cm}
\addplot [color = red, style = thick,mark = square*,  mark size=1.2pt] table [col sep=comma, x index = 0, y index = 1,x filter/.code={\pgfmathparse{mod(\coordindex,50)==0 ? x : inf}}] {TTNN_data/toy_example_avarage_accuracy_cubic.csv};
\addplot [color = blue, style = thick,mark = *,  mark size=1.2pt] table [col sep=comma, x index = 0, y index = 2,x filter/.code={\pgfmathparse{mod(\coordindex,50)==0 ? x : inf}}] {TTNN_data/toy_example_avarage_accuracy_cubic.csv};
\addplot [color = black, style = thick,mark =triangle*,  mark size=1.2pt] table [col sep=comma, x index = 0, y index = 3,x filter/.code={\pgfmathparse{mod(\coordindex,50)==0 ? x : inf}}] {TTNN_data/toy_example_avarage_accuracy_cubic.csv};
\end{semilogyaxis}
\end{tikzpicture}~\begin{tikzpicture}
\begin{semilogyaxis}[
    x label style={at={(axis description cs:0.5,-0.1)},anchor=north},
    xlabel={\small{rank}},
    ylabel={\small{ time (s)}},
    ymajorgrids=true,
    grid style=dashed,
    xtick = {0,16,32,48},
    width=.45\linewidth,
    mark size = 1.2pt,
    restrict x to domain=0:125
]
\hspace*{0.5cm}
\addplot [color = red, style = thick,mark = square*,  mark size=1.2pt] table [col sep=comma, x index = 0, y index = 1,x filter/.code={\pgfmathparse{mod(\coordindex,50)==0 ? x : inf}}] {TTNN_data/toy_example_avarage_times_cubic.csv};
\addplot [color = blue, style = thick,mark = *,  mark size=1.2pt] table [col sep=comma, x index = 0, y index = 2,x filter/.code={\pgfmathparse{mod(\coordindex,50)==0 ? x : inf}}] {TTNN_data/toy_example_avarage_times_cubic.csv};
\addplot [color = black, style = thick,mark =triangle*,  mark size=1.2pt] table [col sep=comma, x index = 0, y index = 3,x filter/.code={\pgfmathparse{mod(\coordindex,50)==0 ? x : inf}}] {TTNN_data/toy_example_avarage_times_cubic.csv};
\end{semilogyaxis}
\end{tikzpicture}\\
\begin{tikzpicture}
\begin{semilogyaxis}[
    x label style={at={(axis description cs:0.5,-0.1)},anchor=north},
    xlabel={\small{rank}},
    ylabel={\small{ $\|\mathcal{T}-\mathcal{T}_r\|_F/\|\mathcal{T}\|_F$}},
    ymajorgrids=true,
    xtick = {0,16,32,48},
    grid style=dashed,
    width=.45\linewidth,
    mark size = 1.2pt,
    restrict x to domain=0:125
]
\hspace*{0.5cm}
\addplot [color = red, style = thick,mark = square*,  mark size=1.2pt] table [col sep=comma, x index = 0, y index = 1,x filter/.code={\pgfmathparse{mod(\coordindex,50)==0 ? x : inf}}] {TTNN_data/toy_example_avarage_accuracy_exp.csv};
\addplot [color = blue, style = thick,mark = *,  mark size=1.2pt] table [col sep=comma, x index = 0, y index = 2,x filter/.code={\pgfmathparse{mod(\coordindex,50)==0 ? x : inf}}] {TTNN_data/toy_example_avarage_accuracy_exp.csv};
\addplot [color = black, style = thick,mark =triangle*,  mark size=1.2pt] table [col sep=comma, x index = 0, y index = 3,x filter/.code={\pgfmathparse{mod(\coordindex,50)==0 ? x : inf}}] {TTNN_data/toy_example_avarage_accuracy_exp.csv};
\end{semilogyaxis}
\end{tikzpicture}~\begin{tikzpicture}
\begin{semilogyaxis}[
    x label style={at={(axis description cs:0.5,-0.1)},anchor=north},
    xlabel={\small{rank}},
    ylabel={\small{ time (s)}},
    ymajorgrids=true,
    xtick = {0,16,32,48},
    grid style=dashed,
    width=.45\linewidth,
    mark size = 1.2pt,
    restrict x to domain=0:125
]
\hspace*{0.5cm}
\addplot [color = red, style = thick,mark = square*,  mark size=1.2pt] table [col sep=comma, x index = 0, y index = 1,x filter/.code={\pgfmathparse{mod(\coordindex,50)==0 ? x : inf}}] {TTNN_data/toy_example_avarage_times_exp.csv};
\addplot [color = blue, style = thick,mark = *,  mark size=1.2pt] table [col sep=comma, x index = 0, y index = 2,x filter/.code={\pgfmathparse{mod(\coordindex,50)==0 ? x : inf}}] {TTNN_data/toy_example_avarage_times_exp.csv};
\addplot [color = black, style = thick,mark =triangle*,  mark size=1.2pt] table [col sep=comma, x index = 0, y index = 3,x filter/.code={\pgfmathparse{mod(\coordindex,50)==0 ? x : inf}}] {TTNN_data/toy_example_avarage_times_exp.csv};
\end{semilogyaxis}
\end{tikzpicture}\\
 \caption{Frobenius error of approximation obtained by the TTNN algorithm with the index tree in Figure \ref{fig: example of index tree} and Khatri-Rao sketchings for different values of $r$ on 6D synthetic tensors of mode size $500$ with different decays: quadratic, cubic and exponential.} 
    \label{fig: TTN rounding}
\end{figure}
Consistent with the results obtained for dense tensors, we observe that, in terms of accuracy, TTNN performs comparably to TTN-HMT, with a relatively constant gap when compared to TTN-SVD.
In terms of runtime, the two randomized algorithms significantly outperform TTN-SVD. Furthermore, we demonstrate that satisfactory accuracy can be achieved using Khatri-Rao embeddings, despite the absence of robust theoretical guarantees.

\section{Conclusions}\label{sec:conclusions}This paper introduced the tree tensor network Nyström (TTNN), a streamable method for the low-rank approximation of a tensor in any tree tensor network format. This method extends existing algorithms based on generalized Nyström, in particular by choosing the proper tree structure it is possible to retrieve the multilinear Nyström \cite{bucci2024multilinear} for the Tucker format and the streaming tensor train approximation \cite{kressner2023streaming} for the tensor train format. The TTNN algorithm preserves the key features of the generalized Nyström for matrices, i.e. it is streamable, randomized, single-pass, and cost-effective.
These properties hold since TTNN avoids the costly orthogonalizations of the hierarchical SVD (and of its randomized version) by computing an approximation based on two-side sketches that allow efficient updates of the approximation after linear updates of the tensor.
\par
We provided accuracy guarantees on the method by proving a deterministic error bound that holds for all dimension reduction matrices (DRMs). While obtaining an error bound in expected value is generally a tough task, the literature contains a wealth of information on standard Gaussian matrices that allowed us to provide an expected error bound for standard Gaussian DRMs.
\par
This paper also introduced the sequential tree tensor network Nystrom (STTNN) approximant, a sequential variant of the TTNN. Sequentiality allows the manipulation of progressively smaller tensors while computing the approximant and is particularly advantageous in the approximation of dense tensors. We have not been able to provide an error bound in expected value due to the complicated structure of the sketches, but it would be another possible avenue of research. Furthermore, we showed how these methods can be applied to the rounding of a tensor given in tree tensor network format. 
\par
All our error bounds are supported by experiments showing the efficiency of TTNN and STTNN. In particular, our methods achieve similar but slightly worse accuracy than the TTN-SVD. This loss of accuracy is compensated by the computational gain given by parallelizability and streamability. Notably, the streamability property has been pivotal in developing the randomized sketched TT-GMRES algorithm \cite{bucci2024randomized} through STTA—a randomized adaptation of the classic TT-GMRES algorithm \cite{dolgov2013tt} for solving linear systems in TT format. An intriguing direction for future research would be to extend these techniques, using TTNN, to develop efficient randomized solvers for linear systems in more general tree tensor formats.
\par
Although the proposed algorithms allow for the approximation of a tensor with any acyclic tensor diagram, certain applications in quantum mechanics benefit from approximations in the tensor ring and MERA formats. These formats are cyclic tensor networks and adapting our methods to these formats could be an interesting challenge.
This work is hence another step towards the extension of streamable algorithms to more complicated tensor networks and could unlock new applications in tensor-based computations.  

\section*{Acknowledgement}
We would like to thank Leonardo Robol and Bart Vandereycken for helpful discussions and pointers to the literature.
\bibliographystyle{unsrt}
\bibliography{references}

\end{document}